\documentclass[12pt]{article}
\usepackage{amssymb,amsfonts,
amsmath,latexsym,amsthm,}
\title{Unitary and Euclidean Representations of a Quiver\footnotetext{This is the author's version of a work that was published in Linear Algebra Appl. 278 (1998) 37--62. Partially supported
        by Grant No. U6E000 from the   International Science Foundation.}
}
\author{Vladimir V. Sergeichuk \\
        Institute of Mathematics,
        Tereshchenkivska 3, Kiev, Ukraine\\ sergeich@imath.kiev.ua}
\date{}

\newtheorem{theorem}{Theorem}[section]
\newtheorem{lemma}{Lemma}[section]
\newtheorem{corollary}{Corollary}[section]
\theoremstyle{remark}

\def\newpic#1{%
   \def\emline##1##2##3##4##5##6{%
      \put(##1,##2){\special{em:point #1##3}}%
      \put(##4,##5){\special{em:point #1##6}}%
      \special{em:line #1##3,#1##6}}}
\newpic{}

\newcommand{\matr}[4]%
   {\left[\genfrac{}{}{0pt}{}{#1}{#3}\, \genfrac{}{}{0pt}{}{#2}{#4}\right]}
\newcommand{\eprf}
{\hfill$\Box$}

\newcommand{\Ker}%
{\mathop{\rm Ker}\nolimits}

\begin{document}
\maketitle

\begin{abstract}
     A unitary (Euclidean) representation of a quiver is given
by assigning to each vertex a
unitary (Euclidean) vector
space and to each arrow a
linear mapping of the
corresponding vector spaces.
We recall an algorithm for
reducing the matrices of a
unitary representation to
canonical form, give a
certain description of the
representations of canonical
form, and reduce the problem
of classifying Euclidean
representations to the
problem of classifying
unitary representations. We
also describe the set of
dimensions of all
indecomposable unitary
(Euclidean) representations
of a quiver and establish the
number of parameters in an
indecomposable unitary
representation of a given
dimension.
\end{abstract}

\newcommand{\quiv}{%
\!\!
\unitlength 0.5mm
\linethickness{0.4pt}
\begin{picture}(9,0)(7,13.5)
\put(15.00,15.00){\circle*{0.7}}
\put(9.47,15.00){\oval(4.00,4.00)[l]}
\bezier{16}(9.53,16.93)(11.33,17.00)(13.73,15.93)
\put(13.80,14.33){\vector(2,1){0.2}}
\bezier{20}(9.53,13.00)(11.80,13.13)(13.80,14.33)
\end{picture}
\! }

\newcommand{\pair}{%
\!
\unitlength 0.5mm
\linethickness{0.4pt}
\begin{picture}(15,0)(-7,-1.5)
\put(0.00,0.00){\circle*{0.8}}
\put(-5.53,0.00){\oval(4.00,4.00)[l]}
\bezier{16}(-5.47,1.93)(-3.67,2.00)(-1.27,0.93)
\put(-1.20,-0.67){\vector(2,1){0.2}}
\bezier{20}(-5.47,-2.00)(-3.20,-1.87)(-1.20,-0.67)
\put(5.53,0.00){\oval(4.00,4.00)[r]}
\bezier{16}(5.47,1.93)(3.67,2.00)(1.27,0.93)
\put(1.20,-0.67){\vector(-2,1){0.2}}
\bezier{20}(5.47,-2.00)(3.20,-1.87)(1.20,-0.67)
\end{picture}
\!\! }

\section{Introduction}

Many problems of linear
algebra can be formulated and
studied in terms of quivers
and their representations,
which were proposed by
Gabriel \cite{a_un1} (see
also \cite{a_un2}). A {\it
quiver} is a directed graph.
Its {\it representation}
${\cal A}$ is given by
assigning to each vertex $i$
a vector space ${\cal A}_i$
and to each arrow $\alpha : i
\to j$  a linear mapping
${\cal A}_{\alpha}: {\cal
A}_i \to {\cal A}_j$. For
example, the canonical form
problems for representations
of the quivers \quiv and \!
\unitlength 0.3mm
\linethickness{0.4pt}
\begin{picture}(22.12,9.55)(2,2)
\put(1.27,4.33){\circle*{0.89}}
\put(21.67,4.33){\circle*{0.89}}
\put(19.67,5.33){\vector(3,-1){0.2}}
\multiput(2.67,4.77)(0.20,0.11){10}{\line(1,0){0.20}}
\multiput(4.69,5.87)(0.30,0.11){7}{\line(1,0){0.30}}
\multiput(6.78,6.66)(0.54,0.12){4}{\line(1,0){0.54}}
\multiput(8.93,7.14)(1.11,0.08){2}{\line(1,0){1.11}}
\multiput(11.15,7.31)(1.14,-0.07){2}{\line(1,0){1.14}}
\multiput(13.43,7.16)(0.59,-0.11){4}{\line(1,0){0.59}}
\multiput(15.78,6.71)(0.32,-0.11){12}{\line(1,0){0.32}}
\put(19.67,3.33){\vector(3,1){0.2}}
\multiput(2.67,3.89)(0.20,-0.11){10}{\line(1,0){0.20}}
\multiput(4.69,2.76)(0.30,-0.12){7}{\line(1,0){0.30}}
\multiput(6.78,1.95)(0.43,-0.10){5}{\line(1,0){0.43}}
\multiput(8.93,1.47)(1.11,-0.08){2}{\line(1,0){1.11}}
\multiput(11.15,1.30)(1.14,0.07){2}{\line(1,0){1.14}}
\multiput(13.43,1.45)(0.59,0.12){4}{\line(1,0){0.59}}
\multiput(15.78,1.92)(0.32,0.12){12}{\line(1,0){0.32}}
\end{picture}
correspond to the canonical
form problems for linear
operators (whose solution is
the Jordan normal form) and
for pairs of linear mappings
from one space to another
(the matrix pencil problem,
solved by Kronecker).

In this chapter we study
unitary and Euclidean
representations of a quiver
up to isometry. A {\it
unitary} ({\it
Euclidean})\label{page_un_ksdudr}
{\it representation} ${\cal
A}$ is given by assigning to
each vertex $i$ a finite
dimensional unitary
(Euclidean) space ${\cal
A}_i$ and to each arrow
$\alpha : i \to j$  a linear
mapping ${\cal A}_{\alpha}:
{\cal A}_i \to {\cal A}_j$.
We say that two unitary
(Euclidean) representations
${\cal A}$ and ${\cal B}$ are
{\it isometric} and write
$\cal A\simeq\cal B$ if there
exists a system of isometries
$\Phi_i :  {\cal A}_i \to
{\cal B}_i$ such that $\Phi_j
{\cal A}_{\alpha} = {\cal
B}_{\alpha}\Phi_i$ for each
$\alpha : i\to j$.

Our main tool is Littlewood's
algorithm \cite{a_un3} for
reducing matrices to
triangular canonical form via
unitary similarity. In
\cite{a_un4} I rediscovered
Littlewood's algorithm and
applied it to the canonical
form problem for unitary
representations of a quiver.
Various algorithms for
reducing matrices to
different canonical forms
under unitary similarity were
also proposed by Brenner,
Mitchell, McRae, Radjavi,
Benedetti and Gragnolini, and
others; see Shapiro's survey
\cite{a_un5}.

In Section \ref{s_un2} we
recall briefly Littlewood's
algorithm and study the
structure of canonical
matrices much as it was made
in [4] for the matrices of
linear operators in a unitary
space.

We say that a matrix problem
is {\it unitarily wild} if it
contains the problem of
classifying linear operators
in a unitary space. In
Section \ref{ss_un2.3} we
show that the last problem
contains the problem of
classifying unitary
representations of an
arbitrary quiver (i.e., it is
hopeless in a certain sense)
and give examples of
unitarily wild matrix
problems.

The vector $$\dim{\cal
A}=({\dim}{\cal A}_1,
{\dim}{\cal A}_2, \dots,
{\dim}{\cal A}_p) \in
{\mathbb N}_0^p$$ is called
the {\it dimension} of a
representation ${\cal A}$ of
a quiver $Q$ with vertices
$1, 2,\dots,p$ (we denote
$${\mathbb N}=\{1,2,\dots\},\quad{\mathbb
N}_0=\{0,1,2,\dots\}).$$
     In Section \ref{s_un3} we describe the set of dimensions of
direct-sum-indecomposable
unitary representations of a
quiver, and  establish the
number of parameters in an
indecomposable unitary
representation of a given
dimension. Analogous, but
much more fundamental and
complicated, results for
non-unitary representations
of a quiver were obtained by
V. G. Kac
\cite{a_un6,a_un7,a_un8} (see
also \cite[Sect 7.4]{a_un2}).

In particular, if $z\in
{\mathbb N}^p$ and $Q$ is a
connected quiver other than
$\bullet$ and
$\bullet\!\to\!\bullet$, then
there exists an
indecomposable unitary
representation of dimension
$z$ if and only if $zM_Q\ge
z$, where $M_Q=[m_{ij}]$ is
the $p\times p$ matrix whose
entry $m_{ij}$ is the number
of arrows of the form $i\to
j$ and $i\gets j$, where
$$(t_1,\dots,t_p)\ge
(z_1,\dots, z_p)$$ means
$$t_1\ge z_1,\dots, t_p\ge
z_p.$$

In Section \ref{s_un4} we
study Euclidean
representations of a quiver.
     Let ${\cal A}^{\mathbb C}$ denote the unitary representation obtained
from  a Euclidean
representation ${\cal A}$ by
complexification (${\cal A}$
and  ${\cal A}^{\mathbb C}$
are given by the same set of
real matrices). In Section
\ref{ss_un4.1} we prove
intuitively obvious facts
that
\begin{itemize}
  \item[(i)] ${\cal A}^{\mathbb
C}\simeq {\cal B}^{\mathbb
C}$ implies ${\cal
A}\simeq{\cal B}$, and
  \item[(ii)] if ${\cal A}$ is
indecomposable and ${\cal
A}^{\mathbb C}$ is
decomposable, then ${\cal
A}^{\mathbb C}\simeq{\cal
U}\oplus {\bar {\cal U}}$,
where ${\cal U}$ is an
indecomposable unitary
representation.
\end{itemize}
This will imply that unitary
and Euclidean representations
have the same sets of
dimensions of indecomposable
representations.

In Section \ref{ss_un4.2} we
study, when a given unitary
representation of a quiver
can be obtained by
complexification. In
particular, let $A$ be a
complex matrix that is not
unitarily similar to a direct
sum of matrices, and let
$S^{-1}AS={\bar A}$ for a
unitary matrix $S$ (such $S$
exists if $A$ is unitarily
similar to a real matrix).
Then $A$ is unitarily similar
to a  real matrix if and only
if $S$ is symmetric.


\section{Unitary matrix problems}
\label{s_un2}

We suppose that the complex
numbers are lexicographically
ordered:
\begin{equation}
a+bi\preceq a'+b'i \ \mbox{if
either}\  a=a' \ \mbox{and} \
b\le b',\  \mbox{or}\ a<a';
\label{un-1.0}
\end{equation}
and that the set of blocks of
a block matrix $A=[A_{ij}]$
are linearly ordered:
             \begin{equation}
 A_{ij}\le A_{i'j'} \ \mbox{if either} \ i=i'  \ \mbox{and} \ j\le j',
 \ \mbox{or}\   i>i'.                              \label{un-1.1}
                  \end{equation}

     A block complex matrix with a given (perhaps empty)
set of marked square blocks
will be called a {\it marked
block matrix}\/; a square
block is marked by a line
along its principal diagonal.
By a {\it unitary matrix
problem} we mean the
classification problem for
marked block matrices
$$A=[A_{ij}],\qquad 1\le i\le
l,\quad 1\le j\le r,$$ up to
transformations
                  \begin{equation}
A\mapsto
B:=R^{-1}AS=[R^{-1}_iA_{ij}S_j],
\label{un-1.2}
                  \end{equation}
where
$$R=R_1\oplus\dots\oplus
R_l,\qquad
S=S_1\oplus\dots\oplus S_r$$
are unitary matrices, and
$R_i=S_j$ whenever the block
$A_{ij}$ is marked. The
transformation (\ref{un-1.2})
is called an {\it admissible
transformation}; we say that
these marked block matrices
$A$ and $B$ (with the same
disposition of marked blocks)
are {\it equivalent}  and
write $A\sim B$ or
                  \begin{equation}
     (R,S):A\leadsto B.
     \label{un-1.3}
                  \end{equation}
Notice that a matrix
consisting of a single block
is reduced by transformations
of unitary similarity if the
block is marked, and by
transformations of unitary
equivalence otherwise.
Moreover, the matrices of
every unitary representation
${\cal A}$ of a quiver can be
placed into a marked block
matrix $A$ such that the
admissible transformations
with $A$ correspond to
reselections of the
orthogonal bases in the
spaces of ${\cal A}$, for
example,
\begin{equation}
\special{em:linewidth 0.4pt}
\unitlength 0.90mm
\linethickness{0.4pt}
\begin{picture}(116.34,20.33)(5,26)
\put(16.34,20.33){\makebox(0,0)[cc]{1}}
\put(46.34,20.33){\makebox(0,0)[cc]{3}}
\put(31.34,35.33){\makebox(0,0)[cc]{2}}
\put(17.84,21.66){\vector(-1,-1){0.2}}
\emline{30.01}{33.83}{1}{17.84}{21.66}{2}
\put(18.34,20.33){\vector(-1,0){0.2}}
\emline{44.34}{20.33}{3}{18.34}{20.33}{4}
\put(31.34,22.83){\makebox(0,0)[cc]{$\nu$}}
\put(22.84,29.33){\makebox(0,0)[cc]{$\mu$}}
\put(40.34,29.33){\makebox(0,0)[cc]{$\xi$}}
\put(3.84,20.33){\makebox(0,0)[cc]{$\lambda$}}
\emline{71.34}{10.33}{5}{71.34}{40.33}{6}
\emline{71.34}{10.33}{7}{116.34}{10.33}{8}
\emline{116.34}{10.33}{9}{116.34}{40.33}{10}
\emline{116.34}{40.33}{11}{71.34}{40.33}{12}
\emline{71.34}{25.33}{13}{116.34}{25.33}{14}
\emline{86.34}{40.33}{15}{86.34}{10.33}{16}
\emline{101.34}{40.33}{17}{101.34}{10.33}{18}
\put(78.84,32.83){\makebox(0,0)[cc]{$A_{\lambda}$}}
\put(93.84,32.83){\makebox(0,0)[cc]{$A_{\mu}$}}
\put(108.84,32.83){\makebox(0,0)[cc]{$A_{\nu}$}}
\put(108.84,17.83){\makebox(0,0)[cc]{0}}
\put(93.84,17.83){\makebox(0,0)[cc]{$A_{\xi}$}}
\put(78.84,17.83){\makebox(0,0)[cc]{0}}
\put(78.84,45.33){\makebox(0,0)[cc]{$S_1$}}
\put(93.84,45.33){\makebox(0,0)[cc]{$S_2$}}
\put(108.84,45.33){\makebox(0,0)[cc]{$S_3$}}
\put(66.34,32.83){\makebox(0,0)[cc]{$S_1^{-1}$}}
\put(66.34,17.83){\makebox(0,0)[cc]{$S_3^{-1}$}}
\emline{71.34}{40.33}{19}{76.34}{35.33}{20}
\emline{81.34}{30.33}{21}{86.34}{25.33}{22}
\emline{101.34}{25.33}{23}{106.34}{20.33}{24}
\emline{111.34}{15.33}{25}{116.34}{10.33}{26}
\emline{116.34}{10.33}{27}{116.34}{10.33}{28}
\put(44.67,21.66){\vector(1,-1){0.2}}
\emline{32.67}{33.66}{29}{44.67}{21.66}{30}
\put(8.37,20.33){\oval(4.07,4.27)[l]}
\emline{8.37}{22.37}{31}{10.59}{22.27}{32}   
\emline{10.59}{22.27}{33}{14.80}{21.20}{34}
\put(15.00,19.67){\vector(4,1){0.2}}
\emline{8.40}{18.20}{35}{9.88}{18.30}{36}
\emline{9.88}{18.30}{37}{15.00}{19.67}{38}
\end{picture}
\label{un-1.4}
                   \end{equation}
\\

\subsection{An algorithm}
\label{ss_un2.1}

The algorithm is based on the
following two lemmas:
\medskip

\begin{lemma}\label{lem_un2.1.1}
{\rm(a)} {Each complex matrix
$A$ is unitarily equivalent
to the matrix}
\begin{equation}
D=a_1I\oplus\dots\oplus
a_{k-1}I\oplus 0,           \
\ \ a_i\in\mathbb R, \ \ \
a_1>\dots>a_{k-1}>0.
\label{un-2.2}
\end{equation}

{\rm(b)} {If $R^{-1}DS=D'$,
where $R$ and $S$ are unitary
matrices and $D,\ D'$ are of
the form} \eqref{un-2.2},
{\it then $D=D'$,
$$S=S_1\oplus\dots \oplus
S_{k-1}\oplus S',$$ and
$$R=S_1\oplus\dots\oplus
S_{k-1}\oplus R',$$ where
each $S_i$ has the same size
as $a_iI$}.
\end{lemma}

\begin{lemma}\label{lem_un2.1.2}
     {\rm(a)} {\it Each square complex matrix $A$ is unitarily
similar to the
block-triangular matrix}
\begin{equation}
F=\left[ \begin{array}{cccc}
             \lambda_1I&  F_{12}   & \cdots & F_{1k}   \\
                       &\lambda_2I & \cdots & F_{2k}   \\
                       &           & \ddots & \vdots  \\
                  0    &           &     & \lambda_kI
              \end{array} \right],\quad
      \parbox{4cm}
             {$\lambda_1\succeq\dots\succeq\lambda_k$
              $($see \eqref{un-1.0}$)$,
              the columns of $F_{i,i+1}$
              are linearly independent
              if $\lambda_i=\lambda_{i+1}.$}
              \label{un-2.3}
\end{equation}

{\rm(b)} {\it If
$S^{-1}FS=F'$, where $S$ is a
unitary matrix and $F$ and
$F'$ have the form}
\eqref{un-2.3}, {\it then
$\lambda_iI=\lambda'_iI$ and
$$S=S_1\oplus\dots\oplus
S_{k},$$ where each $S_i$ has
the same size as
$\lambda_iI$}.
\end{lemma}

\begin{proof} These lemmas were proved in many articles, see,
for example,
\cite{a_un3,a_un4,a_un5}, so
we give only an outline of
their proofs. Part (a) of
Lemma \ref{lem_un2.1.1} is
the singular value
decomposition; part (b)
follows from
$$D=D',\qquad
S^*D^*R^{*-1}=D^*,\qquad
S^{-1}DR=D,$$
$$D^2R=RD^2,\qquad D^2S=SD^2.$$ The
matrix (\ref{un-2.3}) is the
matrix of an arbitrary linear
operator $A: {\mathbb C}^n\to
{\mathbb C}^n$ in an
orthogonal basis
$f_1,\dots,f_n$ such that
$f_1,\dots,f_{t_r}$ is a
basis of
$$\Ker(A-\lambda_1I)\cdots
(A-\lambda_rI),\qquad 1\le
r\le k,$$ where
$$(x-\lambda_1)\cdots
(x-\lambda_k),\qquad
\lambda_1\succeq\dots\succeq\lambda_k,$$
is the minimal polynomial of
$A$; it proves part (a) of
Lemma \ref{lem_un2.1.2}.
Successively equating the
blocks of $FS=SF'$ ordered
with (\ref{un-1.1}), we prove
part (b). \end{proof}

By the {\it canonical
part}\label{page_un2.1.3} of
the matrix (\ref{un-2.2}) or
(\ref{un-2.3}), we mean the
matrix (\ref{un-2.2}) or,
respectively, the collection
of blocks $F_{ij},\ i\ge j.$
According to Lemmas
\ref{lem_un2.1.1} and
\ref{lem_un2.1.2}, the
canonical part is uniquely
determined by the initial
matrix $A$ and does not
change if $A$ is replaced by
a unitarily equivalent or,
respectively, similar matrix.
\medskip

      {\it The algorithm for reducing a marked block matrix $A=[A_{ij}]$
to canonical form:}\\[-3mm]
\label{page_un_154}

     Let $A_{pq}$ be the first (in the ordering
     (\ref{un-1.1}))
block of $A$ that changes
under admissible
transformations
({\ref{un-1.2}). Depending on
the arrangement of the marked
blocks, it is reduced by the
transformations of unitary
equivalence or similarity.
Respectively, we reduce
$A=[A_{ij}]$ to the matrix
$\tilde{A}=[\tilde{A}_{ij}]$
with $\tilde{A}_{pq}$ of the
form (\ref{un-2.2}) or
(\ref{un-2.3}), and then
restrict ourselves to those
admissible transformations
with $\tilde{A}$ that
preserve the canonical part
of $\tilde{A}_{pq}$. As
follows from Lemmas
\ref{lem_un2.1.1}(b) and
\ref{lem_un2.1.2}(b), they
are exactly the admissible
transformations with the
marked block matrix $A'$ that
is obtained in the following
way: The block
$\tilde{A}_{pq}$ of the form
(\ref{un-2.2}) or
(\ref{un-2.3}) consists of
$k$ horizontal and $k$
vertical strips; we extend
this partition to the whole
$p$-th horizontal and the
whole $q$-th vertical strips
of $\tilde{A}$. If new $k$
divisions pass through the
marked block
$\tilde{A}_{ij}$, we carry
out $k$ perpendicular
divisions such that
$\tilde{A}_{ij}$ is
partitioned into $k\times k$
subblocks with  square
diagonal blocks (they are
crossed by the marking line)
and repeat this for all new
divisions. We additionally
mark the subblocks
$a_1I,\dots,a_{k-1}I$ of
$\tilde{A}_{pq}$ if it has
the form (\ref{un-2.2}). The
obtained marked block matrix
$A'$ will be called the {\it
derived matrix} of $A$.
 Clearly, $A\sim B$ implies $A'\sim B'$.

      Let us consider the sequence of derived matrices
              \begin{equation}
  A^{(0)}:=A,\ A',\ A'',\dots, A^{(s)}.
  \label{un-2.3.1}
              \end{equation}
 This sequence ends with a certain matrix $A^{(s)},\ s\ge 0,$
 for which the admissible transformations do not change any of its blocks,
 i.e, $A^{(s)}$ is equivalent only to itself.
Then $A\sim B$ implies
$A^{(s)}\sim B^{(s)}$, i.e.,
$A^{(s)}=B^{(s)}$. Remove
from $A^{(s)}$ all additional
divisions into subblocks and
additional marking lines that
have appeared during the
reduction of $A$ to
$A^{(s)}$. The obtained
marked block matrix will be
called a {\it canonical
matrix} or the {\it canonical
form of} $A$ and will be
denoted by
$A^{\infty}$.\label{page_un2}
We have the following
\medskip

\begin{theorem}\label{th_un2.1.3}
{\it Each marked block matrix
$A$ is equivalent to the
uniquely determined canonical
matrix $A^{\infty}$;
moreover, $A\sim B$ if and
only if
$A^{\infty}=B^{\infty}$.}\eprf
\end{theorem}

                                                         \medskip

We will take under
consideration the null
matrices $0_{0n},\ 0_{m0},$
and $0_{00}$ of size $0\times
n,\ m\times 0,$ and $0\times
0$, putting for a $p\times q$
matrix $M$
$$
M\oplus 0_{0n}=
         \left[ \begin{array}{cc}
                       M &  0_{pn}
                 \end{array} \right],\ \ \
M\oplus 0_{m0}=
         \left[ \begin{array}{c}
                       M        \\
                       0_{mq}
                 \end{array} \right],\ \ \
$$
$0_{m0}\oplus 0_{0n}=0_{mn}.$
Respectively, we will
consider block matrices with
``empty'' horizontal and/or
vertical strips.

     Let $A=[A_{ij}]$ and $B=[B_{ij}]\ (1\le i\le l,\
1\le j\le r)$ be marked block
matrices with the same set of
indices $(i,j)$ of the marked
blocks. By the {\it block
direct sum} of $A$ and $B$ we
mean the marked block matrix
$$
A\uplus B:=[A_{ij}\oplus
B_{ij}]
$$
with the same disposition of
marked blocks. If
$$T_1=(R_1,S_1): A\leadsto
C$$ and $$T_2=(R_2,S_2):
B\leadsto D$$ (see
(\ref{un-1.3})), then $R_1,
R_2$ and, respectively,
$S_1,S_2$ are block diagonal
matrices with $l$ and,
respectively, $r$ diagonal
square blocks, and
$$
T_1\uplus T_2:=(R_1\uplus
R_2, S_1\uplus S_2): A\uplus
B\leadsto C\uplus D.
$$

A marked block matrix $A$ is
said to be {\it
indecomposable} if
\begin{itemize}
  \item[(i)] its
size other than $0\times 0$,
and
  \item[(ii)] $A\sim B\uplus C$
implies that $B$ or $C$ has
size $0\times 0$.
\end{itemize}

For every matrices
$M_1,\dots, M_n,\,N,$ we
define
                  \begin{equation}
(M_1,\dots, M_n)\otimes
N:=(M_1\otimes N,\dots,
M_n\otimes N), \label{un-2.4}
                   \end{equation}
where $M_i\otimes N$ is
obtained from $M_i$ by
replacing its entries $a$
with $aN$.

\begin{theorem}\label{th_un2.1.4}
{\rm(a)} {\it Each marked
block matrix $A$ is
equivalent to a matrix of the
form}
\begin{align*}
B&=(P_1\otimes
I_{m_1})\uplus\dots\uplus(P_t\otimes
I_{m_t})\\
   &\sim\underbrace{P_1\uplus\dots\uplus P_1}_{\mbox{$m_1$ copies}}
\uplus\dots\uplus
\underbrace{P_t\uplus\dots\uplus
P_t}_{\mbox{$m_t$ copies}}\,,
\end{align*}
{\it where $P_1,\dots, P_t$
are nonequivalent
indecomposable marked block
matrices, uniquely determined
up to equivalence $($we may
take $P_1=P_1^{\infty},\dots,
P_t=P_t^{\infty})$, and
$m_1,\dots,m_t$ are uniquely
determined natural numbers.
Every admissible
transformation $T:B\leadsto
B$ that preserves $B$ has the
form
$$
T=({\bf 1}_{P_1}\otimes
U_1)\uplus\dots\uplus({\bf
1}_{P_t}\otimes U_t),
$$
where ${\bf 1}_{P_i}=(I,I):
P_i\leadsto P_i$ is the
identity transformation of
$P_i$, and $U_i$ is a unitary
$m_i\times m_i$ matrix $(1\le
i\le t)$}.

      {\rm(b)} {\it A marked block matrix $A$ of size $\ne 0\times 0$
is indecomposable if and only
if every preserving it
admissible transformation
$T:A\leadsto A$ has the form
$T=a{\bf 1}_A,\ a\in{\mathbb
C},\ |a|=1$.}

      {\rm(c)} {\it A canonical matrix can be reduced to an equivalent block
direct sum of indecomposable
canonical matrices using only
admissible permutations of
rows and columns.}
\end{theorem}

\begin{proof}
(a) We may take
$A=A^{\infty}$. Since
admissible transformations
with $A^{(i)},\ 1\le i\le s$,
(see (\ref{un-2.3.1})) are
exactly the admissible
transformations with $A$ that
preserve the already reduced
part of $A^{(i)}$ (preserve
$A^{(s)}$ if $i=s$), the set
of admissible transformations
with $A^{(s)}$ consists of
all $(R,S): A\leadsto A$. By
(\ref{un-1.2}),
$$R=R_1\oplus\dots\oplus
R_l,\qquad S=
S_1\oplus\dots\oplus S_r,$$
where $l\times r$ is the
number of blocks of $A$.
Since $(R,S): A^{(s)}\leadsto
A^{(s)}$, we have
\begin{equation}
R_i=U_{f(i,1)}\oplus\dots\oplus
U_{f(i,l_i)},\qquad
S_j=U_{g(j,1)}\oplus\dots\oplus
U_{g(j,r_j)}, \label{un-2*}
\end{equation}
where
$$f(i,\alpha),\,g(j,\beta)\in\{1,\dots,t\}$$
and $U_1,\dots,U_t$ are
arbitrary unitary matrices of
fixed sizes. $A^{(s)}$
differs from $A$ only by
additional divisions of its
strips into substrips (and by
additional marking lines). We
transpose substrips within
each strip of $A$ to obtain a
matrix $B\sim A$ such that,
for all $(R,S): B\leadsto B,$
we have (\ref{un-2*}) with
$$f(i,1)\le\dots\le
f(i,l_i),\qquad
g(j,1)\le\dots\le g(j,r_j).$$
Clearly, $B$ satisfies (a).

        (b)$\&$(c) These statemants
         are obvious.
\end{proof}


\subsection{The structure of canonical matrices}
\label{ss_un2.2}

In this section we divide the
set of canonical $m\times n$
matrices into disjoint
subsets of canonical matrices
with the same ``scheme'' (the
number of such schemes is
finite for each size $m\times
n$), and show how to
construct all the canonical
matrices with a given scheme
(for matrices under unitary
similarity this was made
briefly in \cite{a_un4}).
\medskip

We partition a canonical
matrix into zones, which
illustrate the reduction
process.

     Let $A=A^{\infty}$ be a canonical matrix. Then all its derived
matrices (\ref{un-2.3.1})
differ from $A$ only by
additional divisions and
marking lines. Denote by
$P_l\ (0\le l<s)$ the first
block of $A^{(l)}$ that
changes under admissible
transformations (it is
reduced when we construct
$A^{(l+1)}$).

     Let $A^{(l)}_{ij}$ be a block of $A^{(l)}$ such that either $A^{(l)}_{ij}\le
P_l$ or $l=s$. The admissible
transformations with
$A^{(l)}$ induce the unitary
equivalence or similarity
transformations with
$A^{(l)}_{ij}$. Respectively,
$A^{(l)}_{ij}$ has the form
(\ref{un-2.2}) or
(\ref{un-2.3}); we denote by
$Z(A^{(l)}_{ij})$ its
canonical part (see page
\pageref{page_un2.1.3}).
 Defining by induction in $l$, we call $Z(A^{(l)}_{ij})$ by a
{\it zone} and $l$ by its
{\it depth} if either $l=0$
or $Z(A^{(l)}_{ij})$ is not
contained in a zone of depth
$<l$.

     For each zone $Z=Z(A^{(l)}_{ij})$, we put
$\mbox{Bl}(Z):=A^{(l)}_{ij}$
and call $Z$ by an {\it
equivalence} ({\it
similarity}) {\it zone} if
Bl$(Z)$ is transformed by
unitary equivalence
(similarity)
transformations.\label{page_un_uyy}

     Clearly, every canonical matrix $A$ is partitioned into equivalence
and similarity zones; for
example (for a marked block
matrix of the form
\unitlength 1mm
\begin{picture}(10,5)(0,1.5)
\put(0,0){\framebox(10,5)[cc]{}}
\put(5,0){\line(0,1){5}}
\put(0,5){\line(1,-1){5}}
\end{picture}
),\newpage
\[
\unitlength 0.82mm
\begin{picture}(130.00,17.00)(35,33)
{\linethickness{1pt}
\put(70.00,8.00){\framebox(60.00,40.00)[cc]{}}
\put(110.00,48.00){\line(0,-1){40.00}}
}
\put(70,48){\line(1,-1){40}}
\put(90.00,48.00){\line(0,-1){20.00}}
\put(90.00,28.00){\line(1,0){40.00}}
\put(110.00,33.00){\line(1,0){20.00}}
\put(125.00,28.00){\line(0,1){20.00}}
\put(120.00,48.00){\line(0,-1){10.00}}
\put(120.00,38.00){\line(1,0){10.00}}
\put(125.00,43.00){\line(1,0){5.00}}
\put(72.50,45.50){\makebox(0,0)[cc]{$i$}}
\put(77.50,45.50){\makebox(0,0)[cc]{0}}
\put(82.50,45.50){\makebox(0,0)[cc]{0}}
\put(87.50,45.50){\makebox(0,0)[cc]{0}}
\put(92.50,45.50){\makebox(0,0)[cc]{2}}
\put(97.50,45.50){\makebox(0,0)[cc]{0}}
\put(102.50,45.50){\makebox(0,0)[cc]{0}}
\put(107.50,45.50){\makebox(0,0)[cc]{0}}
\put(112.50,45.50){\makebox(0,0)[cc]{3}}
\put(117.50,45.50){\makebox(0,0)[cc]{0}}
\put(122.50,45.50){\makebox(0,0)[cc]{2}}
\put(127.50,45.50){\makebox(0,0)[cc]{$i$}}
\put(127.50,40.50){\makebox(0,0)[cc]{0}}
\put(122.50,40.50){\makebox(0,0)[cc]{0}}
\put(117.50,40.50){\makebox(0,0)[cc]{3}}
\put(112.50,40.50){\makebox(0,0)[cc]{0}}
\put(107.50,40.50){\makebox(0,0)[cc]{0}}
\put(102.50,40.50){\makebox(0,0)[cc]{0}}
\put(97.50,40.50){\makebox(0,0)[cc]{2}}
\put(92.50,40.50){\makebox(0,0)[cc]{0}}
\put(87.50,40.50){\makebox(0,0)[cc]{0}}
\put(82.50,40.50){\makebox(0,0)[cc]{0}}
\put(77.50,40.50){\makebox(0,0)[cc]{$i$}}
\put(72.50,40.50){\makebox(0,0)[cc]{0}}
\put(72.50,35.50){\makebox(0,0)[cc]{0}}
\put(77.50,35.50){\makebox(0,0)[cc]{0}}
\put(82.50,35.50){\makebox(0,0)[cc]{$i$}}
\put(87.50,35.50){\makebox(0,0)[cc]{0}}
\put(92.50,35.50){\makebox(0,0)[cc]{0}}
\put(97.50,35.50){\makebox(0,0)[cc]{0}}
\put(102.50,35.50){\makebox(0,0)[cc]{2}}
\put(107.50,35.50){\makebox(0,0)[cc]{0}}
\put(112.50,35.50){\makebox(0,0)[cc]{0}}
\put(117.50,35.50){\makebox(0,0)[cc]{0}}
\put(122.50,35.50){\makebox(0,0)[cc]{$i$}}
\put(127.50,35.50){\makebox(0,0)[cc]{4}}
\put(72.50,30.50){\makebox(0,0)[cc]{0}}
\put(77.50,30.50){\makebox(0,0)[cc]{0}}
\put(82.50,30.50){\makebox(0,0)[cc]{0}}
\put(87.50,30.50){\makebox(0,0)[cc]{$i$}}
\put(92.50,30.50){\makebox(0,0)[cc]{0}}
\put(97.50,30.50){\makebox(0,0)[cc]{0}}
\put(102.50,30.50){\makebox(0,0)[cc]{0}}
\put(107.50,30.50){\makebox(0,0)[cc]{0}}
\put(112.50,30.50){\makebox(0,0)[cc]{0}}
\put(117.50,30.50){\makebox(0,0)[cc]{0}}
\put(122.50,30.50){\makebox(0,0)[cc]{0}}
\put(127.50,30.50){\makebox(0,0)[cc]{4}}
\put(72.50,25.50){\makebox(0,0)[cc]{0}}
\put(77.50,25.50){\makebox(0,0)[cc]{0}}
\put(82.50,25.50){\makebox(0,0)[cc]{0}}
\put(87.50,25.50){\makebox(0,0)[cc]{0}}
\put(92.50,25.50){\makebox(0,0)[cc]{0}}
\put(97.50,25.50){\makebox(0,0)[cc]{0}}
\put(102.50,25.50){\makebox(0,0)[cc]{0}}
\put(107.50,25.50){\makebox(0,0)[cc]{0}}
\put(112.50,25.50){\makebox(0,0)[cc]{5}}
\put(117.50,25.50){\makebox(0,0)[cc]{0}}
\put(122.50,25.50){\makebox(0,0)[cc]{0}}
\put(127.50,25.50){\makebox(0,0)[cc]{0}}
\put(72.50,20.50){\makebox(0,0)[cc]{0}}
\put(77.50,20.50){\makebox(0,0)[cc]{0}}
\put(82.50,20.50){\makebox(0,0)[cc]{0}}
\put(87.50,20.50){\makebox(0,0)[cc]{0}}
\put(92.50,20.50){\makebox(0,0)[cc]{0}}
\put(97.50,20.50){\makebox(0,0)[cc]{0}}
\put(102.50,20.50){\makebox(0,0)[cc]{0}}
\put(107.50,20.50){\makebox(0,0)[cc]{0}}
\put(112.50,20.50){\makebox(0,0)[cc]{0}}
\put(117.50,20.50){\makebox(0,0)[cc]{5}}
\put(122.50,20.50){\makebox(0,0)[cc]{0}}
\put(127.50,20.50){\makebox(0,0)[cc]{0}}
\put(72.50,15.50){\makebox(0,0)[cc]{0}}
\put(77.50,15.50){\makebox(0,0)[cc]{0}}
\put(82.50,15.50){\makebox(0,0)[cc]{0}}
\put(87.50,15.50){\makebox(0,0)[cc]{0}}
\put(92.50,15.50){\makebox(0,0)[cc]{0}}
\put(97.50,15.50){\makebox(0,0)[cc]{0}}
\put(102.50,15.50){\makebox(0,0)[cc]{0}}
\put(107.50,15.50){\makebox(0,0)[cc]{0}}
\put(112.50,15.50){\makebox(0,0)[cc]{0}}
\put(117.50,15.50){\makebox(0,0)[cc]{0}}
\put(122.50,15.50){\makebox(0,0)[cc]{5}}
\put(127.50,15.50){\makebox(0,0)[cc]{0}}
\put(72.50,10.50){\makebox(0,0)[cc]{0}}
\put(77.50,10.50){\makebox(0,0)[cc]{0}}
\put(82.50,10.50){\makebox(0,0)[cc]{0}}
\put(87.50,10.50){\makebox(0,0)[cc]{0}}
\put(92.50,10.50){\makebox(0,0)[cc]{0}}
\put(97.50,10.50){\makebox(0,0)[cc]{0}}
\put(102.50,10.50){\makebox(0,0)[cc]{0}}
\put(107.50,10.50){\makebox(0,0)[cc]{0}}
\put(112.50,10.50){\makebox(0,0)[cc]{0}}
\put(117.50,10.50){\makebox(0,0)[cc]{0}}
\put(122.50,10.50){\makebox(0,0)[cc]{0}}
\put(127.50,10.50){\makebox(0,0)[cc]{3}}
\put(65.00,28.00){\makebox(0,0)[cc]{=}}
\put(60.00,28.5){\makebox(0,0)[cc]{$A$}}  
\end{picture}
\]\\[20mm]
\begin{equation}
\unitlength 0.82mm
\linethickness{0.4pt}
\begin{picture}(130.00,17.00)(20,33)
\put(70.00,8.00){\framebox(60.00,40.00)[cc]{}}
\put(110.00,48.00){\line(0,-1){40.00}}
\put(90.00,28.00){\line(1,0){40.00}}
\put(110.00,33.00){\line(1,0){20.00}}
\put(125.00,28.00){\line(0,1){20.00}}
\put(120.00,48.00){\line(0,-1){10.00}}
\put(120.00,38.00){\line(1,0){10.00}}
\put(125.00,43.00){\line(1,0){5.00}}
\put(127.50,45.50){\makebox(0,0)[cc]{\it
8}}
\put(127.50,40.50){\makebox(0,0)[cc]{\it
8}}
\put(127.50,35.50){\makebox(0,0)[cc]{\it
6}}
\put(127.50,30.50){\makebox(0,0)[cc]{\it
4}}
\put(65.00,28.00){\makebox(0,0)[cc]{=}}
\put(90.00,28.00){\line(0,1){20.00}}
\put(80.00,18.00){\makebox(0,0)[cc]{\it
1}}
\put(120.00,18.00){\makebox(0,0)[cc]{\it
2}}
\put(100.00,38.00){\makebox(0,0)[cc]{\it
3}}
\put(118.33,30.67){\makebox(0,0)[cc]{\it
4}}
\put(115.33,36.33){\makebox(0,0)[cc]{\it
5}}
\put(122.33,43.00){\makebox(0,0)[cc]{\it
7}}
\end{picture}
\label{un-2.7}
\end{equation}
\\[15mm]
\noindent is partitioned into
10 zones, their depths are
indicated on the right of
(\ref{un-2.7}).
\medskip

Let $A$ be a canonical matrix
partitioned into zones. For
each similarity zone, we
replace all its diagonal
elements by stars. For each
equivalence zone, we replace
all its nonzero elements by
circles, and join with a line
its circles corresponding to
equal elements (this line
does not coincide with a
marking line because the
marking lines connect stars).
The other elements of $A$ are
zeros, we replace theirs by
points. The obtained picture
will be called the {\it
scheme} ${\cal S}(A)$ of $A$.

    For example, the canonical matrix
    (\ref{un-2.7}) has the scheme
$$
\unitlength 1.00mm
\linethickness{0.4pt}
\begin{picture}(130.00,40.17)(40,10)
{\linethickness{1pt}
\put(70.00,10.17){\framebox(60.00,40.00)[cc]{}}
\put(110.00,50.17){\line(0,-1){40.00}}
}
\put(90.00,50.17){\line(0,-1){20.00}}
\put(90.00,30.17){\line(1,0){40.00}}
\put(110.00,35.17){\line(1,0){20.00}}
\put(125.00,30.17){\line(0,1){20.00}}
\put(120.00,50.17){\line(0,-1){10.00}}
\put(120.00,40.17){\line(1,0){10.00}}
\put(125.00,45.17){\line(1,0){5.00}}
\put(72.50,47.67){\makebox(0,0)[cc]{$\star$}}
\put(77.50,47.67){\makebox(0,0)[cc]{$\cdot$}}
\put(82.50,47.67){\makebox(0,0)[cc]{$\cdot$}}
\put(87.50,47.67){\makebox(0,0)[cc]{$\cdot$}}
\put(97.50,47.67){\makebox(0,0)[cc]{$\cdot$}}
\put(102.50,47.67){\makebox(0,0)[cc]{$\cdot$}}
\put(107.50,47.67){\makebox(0,0)[cc]{$\cdot$}}
\put(117.50,47.67){\makebox(0,0)[cc]{$\cdot$}}
\put(122.50,47.67){\makebox(0,0)[cc]{$\bullet$}}
\put(127.50,47.67){\makebox(0,0)[cc]{$\star$}}
\put(127.50,42.67){\makebox(0,0)[cc]{$\cdot$}}
\put(122.50,42.67){\makebox(0,0)[cc]{$\cdot$}}
\put(112.50,42.67){\makebox(0,0)[cc]{$\cdot$}}
\put(107.50,42.67){\makebox(0,0)[cc]{$\cdot$}}
\put(102.50,42.67){\makebox(0,0)[cc]{$\cdot$}}
\put(97.50,42.67){\makebox(0,0)[cc]{$\bullet$}}
\put(92.50,42.67){\makebox(0,0)[cc]{$\cdot$}}
\put(87.50,42.67){\makebox(0,0)[cc]{$\cdot$}}
\put(82.50,42.67){\makebox(0,0)[cc]{$\cdot$}}
\put(77.50,42.67){\makebox(0,0)[cc]{$\star$}}
\put(72.50,42.67){\makebox(0,0)[cc]{$\cdot$}}
\put(72.50,37.67){\makebox(0,0)[cc]{$\cdot$}}
\put(77.50,37.67){\makebox(0,0)[cc]{$\cdot$}}
\put(82.50,37.67){\makebox(0,0)[cc]{$\star$}}
\put(87.50,37.67){\makebox(0,0)[cc]{$\cdot$}}
\put(92.50,37.67){\makebox(0,0)[cc]{$\cdot$}}
\put(97.50,37.67){\makebox(0,0)[cc]{$\cdot$}}
\put(102.50,37.67){\makebox(0,0)[cc]{$\bullet$}}
\put(107.50,37.67){\makebox(0,0)[cc]{$\cdot$}}
\put(112.50,37.67){\makebox(0,0)[cc]{$\cdot$}}
\put(117.50,37.67){\makebox(0,0)[cc]{$\cdot$}}
\put(122.50,37.67){\makebox(0,0)[cc]{$\star$}}
\put(72.50,32.67){\makebox(0,0)[cc]{$\cdot$}}
\put(77.50,32.67){\makebox(0,0)[cc]{$\cdot$}}
\put(82.50,32.67){\makebox(0,0)[cc]{$\cdot$}}
\put(87.50,32.67){\makebox(0,0)[cc]{$\star$}}
\put(92.50,32.67){\makebox(0,0)[cc]{$\cdot$}}
\put(97.50,32.67){\makebox(0,0)[cc]{$\cdot$}}
\put(102.50,32.67){\makebox(0,0)[cc]{$\cdot$}}
\put(107.50,32.67){\makebox(0,0)[cc]{$\cdot$}}
\put(112.50,32.67){\makebox(0,0)[cc]{$\cdot$}}
\put(117.50,32.67){\makebox(0,0)[cc]{$\cdot$}}
\put(122.50,32.67){\makebox(0,0)[cc]{$\cdot$}}
\put(72.50,27.67){\makebox(0,0)[cc]{$\cdot$}}
\put(77.50,27.67){\makebox(0,0)[cc]{$\cdot$}}
\put(82.50,27.67){\makebox(0,0)[cc]{$\cdot$}}
\put(87.50,27.67){\makebox(0,0)[cc]{$\cdot$}}
\put(97.50,27.67){\makebox(0,0)[cc]{$\cdot$}}
\put(102.50,27.67){\makebox(0,0)[cc]{$\cdot$}}
\put(107.50,27.67){\makebox(0,0)[cc]{$\cdot$}}
\put(112.50,27.67){\makebox(0,0)[cc]{$\bullet$}}
\put(117.50,27.67){\makebox(0,0)[cc]{$\cdot$}}
\put(122.50,27.67){\makebox(0,0)[cc]{$\cdot$}}
\put(127.50,27.67){\makebox(0,0)[cc]{$\cdot$}}
\put(72.50,22.67){\makebox(0,0)[cc]{$\cdot$}}
\put(77.50,22.67){\makebox(0,0)[cc]{$\cdot$}}
\put(82.50,22.67){\makebox(0,0)[cc]{$\cdot$}}
\put(87.50,22.67){\makebox(0,0)[cc]{$\cdot$}}
\put(92.50,22.67){\makebox(0,0)[cc]{$\cdot$}}
\put(102.50,22.67){\makebox(0,0)[cc]{$\cdot$}}
\put(107.50,22.67){\makebox(0,0)[cc]{$\cdot$}}
\put(112.50,22.67){\makebox(0,0)[cc]{$\cdot$}}
\put(117.50,22.67){\makebox(0,0)[cc]{$\bullet$}}
\put(122.50,22.67){\makebox(0,0)[cc]{$\cdot$}}
\put(127.50,22.67){\makebox(0,0)[cc]{$\cdot$}}
\put(72.50,17.67){\makebox(0,0)[cc]{$\cdot$}}
\put(77.50,17.67){\makebox(0,0)[cc]{$\cdot$}}
\put(82.50,17.67){\makebox(0,0)[cc]{$\cdot$}}
\put(87.50,17.67){\makebox(0,0)[cc]{$\cdot$}}
\put(92.50,17.67){\makebox(0,0)[cc]{$\cdot$}}
\put(97.50,17.67){\makebox(0,0)[cc]{$\cdot$}}
\put(107.50,17.67){\makebox(0,0)[cc]{$\cdot$}}
\put(112.50,17.67){\makebox(0,0)[cc]{$\cdot$}}
\put(117.50,17.67){\makebox(0,0)[cc]{$\cdot$}}
\put(122.50,17.67){\makebox(0,0)[cc]{$\bullet$}}
\put(127.50,17.67){\makebox(0,0)[cc]{$\cdot$}}
\put(72.50,12.67){\makebox(0,0)[cc]{$\cdot$}}
\put(77.50,12.67){\makebox(0,0)[cc]{$\cdot$}}
\put(82.50,12.67){\makebox(0,0)[cc]{$\cdot$}}
\put(87.50,12.67){\makebox(0,0)[cc]{$\cdot$}}
\put(92.50,12.67){\makebox(0,0)[cc]{$\cdot$}}
\put(97.50,12.67){\makebox(0,0)[cc]{$\cdot$}}
\put(102.50,12.67){\makebox(0,0)[cc]{$\cdot$}}
\put(112.50,12.67){\makebox(0,0)[cc]{$\cdot$}}
\put(117.50,12.67){\makebox(0,0)[cc]{$\cdot$}}
\put(122.50,12.67){\makebox(0,0)[cc]{$\cdot$}}
\put(127.50,12.67){\makebox(0,0)[cc]{$\bullet$}}
\put(112.50,47.67){\makebox(0,0)[cc]{$\star$}}
\put(92.50,47.67){\makebox(0,0)[cc]{$\bullet$}}
\put(92.50,27.67){\makebox(0,0)[cc]{$\star$}}
\put(97.50,22.67){\makebox(0,0)[cc]{$\star$}}
\put(102.50,17.67){\makebox(0,0)[cc]{$\star$}}
\put(107.50,12.67){\makebox(0,0)[cc]{$\star$}}
\put(117.50,42.67){\makebox(0,0)[cc]{$\star$}}
\put(127.50,37.67){\makebox(0,0)[cc]{$\bullet$}}
\put(127.50,32.67){\makebox(0,0)[cc]{$\bullet$}}
\put(65.00,30.17){\makebox(0,0)[cc]{=}}
\put(60.00,30.17){\makebox(0,0)[rc]{${\cal
S}(A)$}}
\multiput(70.00,50.17)(0.12,-0.12){334}{\line(0,-1){0.12}}
\multiput(92.50,47.67)(0.12,-0.12){84}{\line(0,-1){0.12}}
\multiput(112.50,27.67)(0.12,-0.12){84}{\line(0,-1){0.12}}
\end{picture}
$$                                

\begin{theorem}\label{th_un2.2.2}
{Each canonical matrix
$A=[a_{ij}]$ with a given
scheme ${\cal S}= [s_{ij}]$
can be constructed by
successive filling of its
zones by numbers starting
with the zones of greatest
depth as follows: Let $Z$ be
a zone of depth $d(Z)$ and
let all entries in zones of
depth $>d(Z)$ be replaced by
numbers. Then we replace all
points, circles, and stars of
$Z$, respectively, by zeros,
positive real numbers, and
complex numbers such that the
following conditions hold:

    $1)$ Let $s_{ij}$ and $s_{i+1,j+1}$ be circles in $Z$. Then
$a_{ij}=a_{i+1,j+1}$ if
$s_{ij}$ and $s_{i+1,j+1}$
are linked by a line, and
$a_{ij}>a_{i+1,j+1}$
otherwise.

$2)$ Let
$s_{\alpha,\beta},\dots,s_{\alpha+k,\beta+k}$
be all stars of $Z$ that lie
under a certain stair of $Z$.
Then
$$a_{\alpha,\beta}=\dots=a_{\alpha+k,\beta+k}.$$
If
$$s_{\alpha+k+1,\beta+k+1},\,\dots,\,s_{\alpha+t,\beta+t}$$
are all stars of $Z$ that lie
under the next stair of $Z$,
then $a_{\alpha,\beta}\succeq
a_{\alpha+t,\beta+t}$;
moreover,
$a_{\alpha,\beta}\succ
a_{\alpha+t,\beta+t}$
whenever the columns of the
block $$[\,
a_{ij}\,|\,\alpha\le i\le
\alpha+k,\beta+k+1\le j\le
\beta+t]$$ are linearly
dependent $($this block has
been filled by numbers
because all its entries are
located in zones of depth
$>d(Z))$.}\eprf \end{theorem}

                                            \medskip

     This theorem gives a convenient way to present solutions of
unitary matrix problems in
small sizes by their sets of
schemes. Thus, the list of
schemes of canonical $5\times
5$ matrices under unitary
similarity was obtained by
Klimenko \cite{a_un11}.


\subsection{Unitarily wild matrix problems}
\label{ss_un2.3}
     The canonical form problem for pairs of $n\times n$ matrices
under simultaneous similarity
(i.e., for representations of
the quiver \pair ) plays a
special role in the theory of
(non-unitary) matrix
problems. It may be proved
that its solution implies the
clasification of
representations of every
quiver (and even
representations of every
finite dimensional algebra).
For this reason, the
classification problem for
pairs of matrices under
simultaneous similarity is
used as a yardstick of the
complexity; Donovan and
Freislich \cite{a_un12} (see
also \cite{a_un2}) suggested
to name a classification
problem {\it wild} if it
contains the problem of
simultaneous similarity, and
otherwise to name it {\it
tame} (in accordance with the
partition of animals into
wild and tame ones).

The canonical form problem
for an $n\times n$ matrix
under unitary similarity
(i.e., for unitary
representations of the
quiver$\! \!$
\unitlength 0.5mm
\linethickness{0.4pt}
\begin{picture}(9,0)(7,13.5)
\put(15.00,15.00){\circle*{0.8}}
\put(9.47,15.00){\oval(4.00,4.00)[l]}
\bezier{16}(9.53,16.93)(11.33,17.00)(13.73,15.93)
\put(13.80,14.33){\vector(2,1){0.2}}
\bezier{20}(9.53,13.00)(11.80,13.13)(13.80,14.33)
\end{picture}
\!) plays the same role in
the theory of unitary matrix
problems: it contains the
problem of classifying
unitary representations of
every quiver. For example,
the problem of classifying
unitary representations of
the quiver ({\ref{un-1.4})
can be regarded (by Lemma
\ref{lem_un2.1.2}) as the
problem of classifying, up to
unitary similarity, matrices
of the form:
        $$ \left[
              \begin{tabular}{c|c|c|c|c}
                  $5I$&$I$&$A_{\lambda}$&$A_{\nu}$&$A_{\mu}$\\   \hline
                  0&$4I$&$I$&0&0\\       \hline
                  0&0& $3I$  &0&0\\       \hline
                  0&0&0& $2I$ &$A_{\xi}$ \\      \hline
                  0&0&0&0& $I$
              \end{tabular}\right].$$
A matrix problem is called
{\it unitarily wild} (or
*-wild, see \cite{a_un13}) if
it contains the problem of
classifying matrices via
unitary similarity, and {\it
unitarily tame} otherwise.

       For each unitary problem, one has an alternative: to solve it
or to prove that it is
unitarily wild (and hence is
hopeless in a certain sense).
In this section we give some
examples of such
alternatives.
\medskip

(i)  Let us consider the
problems of classifying
nilpotent linear operators
$\varphi,\ \varphi^n=0$, in a
unitary space.

For $n=2$ this problem is
unitarily tame; the canonical
matrix of  $\varphi$ (see
page \pageref{page_un2}) is
\[
\begin{bmatrix}
  0 & D\\ 0  & 0
\end{bmatrix},
\]
where $D$ is of the form
({\ref{un-2.2}) without zero
columns. Indeed, a matrix
$F=[F_{ij}]$ of the form
(\ref{un-2.3}) satisfies
$F^2=0$ only if $k=2,\
F_{11}=0,$ and $F_{22}=0$; we
can reduce $F_{12}$ to the
form (\ref{un-2.2}).

      For $n>2$ this problem is unitarily wild since the matrices
    $$ \left[ \begin{array}{ccc}
                       0 &  I   &  X            \\
                       0 &  0   &  I            \\
                       0 &  0   &  0
                 \end{array} \right]
  \ \ \ {\rm and} \ \ \
         \left[ \begin{array}{ccc}
                       0 &  I   &  Y            \\
                       0 &  0   &  I            \\
                       0 &  0   &  0
                 \end{array} \right]$$
are unitarily similar if and
only if $X$ and $Y$ are
unitarily similar (see also
\cite{a_un14}).
\medskip

(ii)  Let us consider the
problem of classifying
$m$-tuples $(p_1,\dots,p_m)$
of projectors $p_i^2=p_i$ in
a unitary space.

For $m=1$ this problem is
unitarily time; the canonical
matrix of a projector $p=p^2$
was obtained in \cite{a_un15}
and \cite{a_un16}. Of course,
it is
\[
\begin{bmatrix}
  I & D\\ 0  & 0
\end{bmatrix},
\]
where $D$ is of the form
({\ref{un-2.2}), since a
matrix $F=[F_{ij}]$ of the
form (\ref{un-2.3}) satisfies
$F^2=F$ only if $k=2,\
F_{11}=I,$ and $F_{22}=0$.

As was proved in [16], for
$m\ge 2$ this problem is
unitarily wild even if $p_1$
is an orthoprojector, i.e.
$p_1=p_1^2=p_1^*$, (since the
pairs of idempotent matrices
     $$ \left(\left[ \begin{array}{cc}
                       I &  0                   \\
                       0 &  0
              \end{array} \right],
              \left[ \begin{array}{cc}
                       X & I-X                  \\
                       X & I-X
              \end{array} \right]\right)
$$
and
$$
        \left(\left[ \begin{array}{cc}
                       I &  0                   \\
                       0 &  0
              \end{array} \right],
              \left[ \begin{array}{cc}
                       Y & I-Y                  \\
                       Y & I-Y
              \end{array} \right]\right)
$$
are unitarily similar if and
only if $X$ and $Y$ are
unitarily similar), or if
$p_1p_2=p_2p_1=0$.
\medskip

(iii) The problems of
classifying the following
operators and systems of
operators in unitary spaces
are unitarily wild:

\begin{itemize}
  \item
Pairs of linear operators
$(\varphi, \psi)$ such that
$$\varphi^2=\psi^2=\varphi\psi=\psi\varphi=0$$
since
     $$ \left(\left[ \begin{array}{cc}
                       0 &  I                   \\
                       0 &  0
              \end{array} \right],
              \left[ \begin{array}{cc}
                       0 &  X                   \\
                       0 &  0
              \end{array} \right]\right)
\ \ \ {\rm and} \ \ \
        \left(\left[ \begin{array}{cc}
                       0 &  I                   \\
                       0 &  0
              \end{array} \right],
              \left[ \begin{array}{cc}
                       0 &  Y                   \\
                       0 &  0
              \end{array} \right]\right)$$
are unitarily similar if and
only if $X$ and $Y$ are
unitarily similar.

  \item
Pairs of selfadjoint
operators  $(\varphi, \psi)$
because
 $ \varphi+i\psi $ is an arbitrary operator. The tame-wild dichotomy
for satisfying quadratic
relation pairs of selfadjoint
operators in a Hilbert space
was studied in [17].

  \item
Pairs of unitary operators
$(\varphi, \psi)$ since
$$(i(\varphi+{\bf
1})(\varphi-{\bf 1})^{-1},\,
i(\psi+{\bf 1})(\psi-{\bf
1})^{-1})$$ is a pair of
selfadjoint operators (the
Cayley transformation).

  \item
Partial isometries (i.e.,
linear operators $\varphi$
such that
$(\varphi^*\varphi)^2=\varphi^*\varphi$),
it was proved in [18].
\end{itemize}
\medskip

(iv) The problem of
classifying unitary
representations of a
connected quiver $Q$ is
unitarily tame if $Q\in
\{\bullet,\ \bullet\! \to
\!\bullet\}$  and unitarily
wild otherwise.

     Indeed, the classification of unitary representations of the quiver
$\bullet\! \to \!\bullet$ is
given by the singular value
decomposition (Lemma
\ref{lem_un2.1.1}).

     The problem of classifying unitary representations of the quiver
$\bullet\! \to \!\bullet\!
\gets \!\bullet$ is unitarily
wild \label{page_un3} because
it reduces to the unitary
matrix problem for marked
block matrices of the form $
\begin{tabular}{|c|c|}
\hline
                              &   \\ \hline
 \end{tabular} $\ ,
and two block matrices

\unitlength 1mm
\linethickness{0.4pt}
\begin{picture}(65.00,22.00)(-15,-3)
\put(0.00,0.00){\framebox(25.00,15.00)[cc]{}}
\put(15.00,15.00){\line(0,-1){15.00}}
\multiput(24,5.00)(-1,0){9}{\circle*{0.3}}        
\multiput(24,10.00)(-1,0){9}{\circle*{0.3}}        
\multiput(20,14.00)(0,-1){9}{\circle*{0.3}}        
\put(2.50,12.50){\makebox(0,0)[cc]{$2I$}}
\put(7.50,12.50){\makebox(0,0)[cc]{0}}
\put(12.50,12.50){\makebox(0,0)[cc]{0}}
\put(17.50,12.50){\makebox(0,0)[cc]{$I$}}
\put(22.50,12.50){\makebox(0,0)[cc]{$X$}}
\put(2.50,7.50){\makebox(0,0)[cc]{0}}
\put(7.50,7.50){\makebox(0,0)[cc]{$I$}}
\put(12.50,7.50){\makebox(0,0)[cc]{0}}
\put(17.50,7.50){\makebox(0,0)[cc]{$I$}}
\put(22.50,7.50){\makebox(0,0)[cc]{$I$}}
\put(2.50,2.50){\makebox(0,0)[cc]{0}}
\put(7.50,2.50){\makebox(0,0)[cc]{0}}
\put(12.50,2.50){\makebox(0,0)[cc]{0}}
\put(17.50,2.50){\makebox(0,0)[cc]{$I$}}
\put(22.50,2.50){\makebox(0,0)[cc]{0}}
\put(32.50,7.50){\makebox(0,0)[cc]{and}}
\put(40.00,0.00){\framebox(25.00,15.00)[cc]{}}
\put(55.00,15.00){\line(0,-1){15.00}}
\multiput(64,5.00)(-1,0){9}{\circle*{0.3}}        
\multiput(56,10.00)(1,0){9}{\circle*{0.3}}        
\multiput(60,14.00)(0,-1){9}{\circle*{0.3}}        
\put(42.50,12.50){\makebox(0,0)[cc]{$2I$}}
\put(47.50,12.50){\makebox(0,0)[cc]{0}}
\put(52.50,12.50){\makebox(0,0)[cc]{0}}
\put(57.50,12.50){\makebox(0,0)[cc]{$I$}}
\put(62.50,12.50){\makebox(0,0)[cc]{$Y$}}
\put(42.50,7.50){\makebox(0,0)[cc]{0}}
\put(47.50,7.50){\makebox(0,0)[cc]{$I$}}
\put(52.50,7.50){\makebox(0,0)[cc]{0}}
\put(57.50,7.50){\makebox(0,0)[cc]{$I$}}
\put(62.50,7.50){\makebox(0,0)[cc]{$I$}}
\put(42.50,2.50){\makebox(0,0)[cc]{0}}
\put(47.50,2.50){\makebox(0,0)[cc]{0}}
\put(52.50,2.50){\makebox(0,0)[cc]{0}}
\put(57.50,2.50){\makebox(0,0)[cc]{$I$}}
\put(62.50,2.50){\makebox(0,0)[cc]{0}}
\end{picture}

\noindent are equivalent if
and only if $X$ and $Y$ are
unitarily similar. We can
change the direction of an
arrow in a quiver by
replacing in each
representation the
corresponding linear mapping
by the adjoint one.
                                                                \medskip

(v)  Let us consider the
problem of classifying
$n$-tuples $(V_1,\dots,V_n)$
of subspaces of a unitary
space $U$ up to the following
equivalence:
$$(V_1,\dots,V_n)\sim
(V_1',\dots,V_n')$$ if there
exists an isometry $\varphi:
U\to U$ such that $$\varphi
V_1=V_1',\dots, \varphi
V_n=V_n'.$$ Fixing an
orthogonal basis in $U$ and
(non-orthogonal) bases in
$V_1,\dots,V_n$, we reduce it
to the canonical form problem
for block matrices
$A=[A_1|\dots|A_n]$ (the
columns of $A_i$ are the
basis vectors of $V_i$ and
hence are linearly
independent) up to unitary
transformations of rows of
$A$ and elementary
(non-unitary) transformations
of columns of $A_i$
$(i=1,\dots,n)$. If $n=1$,
then $A=[A_1] $ reduces to
$I\oplus 0_{p0}$; this
follows from Lemma
\ref{lem_un2.1.1}.

If $n=2$, then $A=[A_1|A_2]$
reduces to
\begin{center}
\unitlength 1.00mm
\linethickness{0.4pt}
\begin{picture}(23.00,20.00)
\put(10.00,0.00){\line(0,1){20.00}}
\put(5.00,15.00){\makebox(0,0)[cc]{{\LARGE\sl
I}}}
\put(5.00,5.00){\makebox(0,0)[cc]{{\LARGE
0}}}
\put(12.50,17.50){\makebox(0,0)[cc]{$I$}}
\put(17.50,17.50){\makebox(0,0)[cc]{0}}
\put(12.50,12.50){\makebox(0,0)[cc]{0}}
\put(17.50,12.50){\makebox(0,0)[cc]{$D$}}
\put(12.50,7.50){\makebox(0,0)[cc]{0}}
\put(17.50,7.50){\makebox(0,0)[cc]{$I$}}
\put(12.50,2.50){\makebox(0,0)[cc]{$0$}}
\put(17.50,2.50){\makebox(0,0)[cc]{0}}
\multiput(15,11)(0,1){9}{\circle*{0.3}}        
\put(15.00,20.00){\line(0,1){0.00}}
\multiput(16,15)(1,0){4}{\circle*{0.3}}        
\put(20.00,20.00){\line(0,-1){20.00}}
\put(0.00,20.00){\line(1,0){20.00}}
\put(20.00,20.00){\line(0,1){0.00}}
\multiput(11,10)(1,0){9}{\circle*{0.3}}        
\put(23.00,10.00){\makebox(0,0)[cc]{,}}
\put(0.00,0.00){\line(0,1){20.00}}
\put(0.00,0.00){\line(1,0){20.00}}
\end{picture}
\end{center}

\noindent where $D$ is of the
form (\ref{un-2.2}). This
block matrix reduces to a
block direct sum of matrices
$$\left[\genfrac{}{}{0pt}{}{1}{0} \left|
\genfrac{}{}{0pt}{}{\alpha}{1}\right]\right.\
(\alpha>0),\quad  [1|1],\quad
[1|0_{10}],\quad
[0_{10}|1],\quad
[0_{10}|0_{10}].$$ (The
problem of classifying pairs
of subspaces in a complex or
real vector space with scalar
product given by a symmetric,
or skew-symmetric, or
Hermitian form was solved in
[19].)

     For $n=3$ this problem is unitarily wild even if we restrict
our consideration to the
triples $(V_1,V_2,V_3)$ with
$V_1 \perp V_2$ since
$$
                  \left[ \begin{tabular}{c|c|c}
                         $I$&$0$&$X$ \\
                         $0$&$I$&$Y$  \\
                         $0$&$0$&$I$
                          \end{tabular} \right]
 \ \ \text{reduces to} \  \
                  \left[ \begin{tabular}{c|c|c}
                         $I$&$0$&$X'$ \\
                         $0$&$I$&$Y'$  \\
                         $0$&$0$&$I$
                          \end{tabular} \right]
$$
if and only if $(X,Y)$ and
$(X',Y')$ determine isometric
unitary representations of
the quiver $\bullet\! \gets
\bullet\! \to \!\bullet$ (see
page \pageref{page_un3}). An
analogous statement was
proved in [20] and [13]: the
problem of classifying
triples $(p_1,p_2,p_3)$ of
orthoprojectors
$p_i=p_i^2=p_i^*$ in a
unitary space is unitarily
wild even if
$$p_1p_2=p_2p_1=0;$$ such a
triple determines, in
one-to-one manner, a triple
$(V_1,V_2,V_3)$ with $V_1
\perp V_2$ by means of
$V_i={\rm Im}\,p_i$.


\section{Unitary representations of a quiver}
\label{s_un3}
     From now on, $Q$ denotes a quiver with vertices $1,\dots,p$
and arrows
$\alpha_1,\dots,\alpha_q$. A
{\it unitary representation}
of  dimension
$$d=(d_1,\dots,d_p)\in
{\mathbb N}_0^p$$ (in short,
a {\it unitary $d$
representation}) will be
given by assigning a matrix
$A_{\alpha}\in {\mathbb
C}^{d_j\times d_i}$ to each
arrow $\alpha: i\to j$, i.e.,
by the sequence
$$A=(A_{\alpha_1},\dots,
A_{\alpha_q})$$ (assigning to
each vertex $i$ the unitary
vector  space $ {\mathbb
C}^{d_i}$ with scalar product
$$(x,y)={\bar x}_1
y_1+\dots+{\bar
x}_{d_i}y_{d_i},$$ we obtain
a unitary representation; see
page
\pageref{page_un_ksdudr}).

     An {\it isometry}
$A\stackrel{\sim}{\to}B$ of
$d$ representations $A$ and
$B$ (an {\it autometry} if
$A=B$) is given by a sequence
$$S=(S_1,\dots,S_p)$$ of
unitary $ d_i\times d_i$
matrices $S_i$ such that $$
S_jA_{\alpha}=B_{\alpha}S_i
$$ for each arrow $\alpha:
i\to j$; we say also that $B$
is obtained  from $A$ by {\it
admissible transformations}
and write $A \simeq B$. An
autometry
 $S:A{\tilde \to} A$ is {\it scalar} if $S=a{\bf 1}_A$,
where $$a\in {\mathbb
C},\quad {\bf
1}_A=(I_{d_1},\dots,
I_{d_p}).$$

For two sequences of matrices
$$M=(M_1,\dots,M_t),\qquad
N=(N_1,\dots,N_t),$$ we
denote $$M\oplus N=(M_1\oplus
N_1,\dots, M_t\oplus N_t).$$
A unitary $d$ representation
$A$ of $Q$ is {\it
indecomposable} if (i)
$d\ne(0,\dots,0)$ and (ii)
$A\simeq B\oplus C$ implies
$B$ or $C$ has dimension
$(0,\dots,0)$.


\subsection{Canonical representations}       
\label{ss_un3.1}

Let $A$ be a unitary
\label{page_un_26468}
representation of $Q$. Using
the algorithm from page
\pageref{page_un_154}, we
reduce $A_{\alpha_1}$ to its
canonical form
 $A_{\alpha_1}^{\infty}$, then restrict the set of admissible
transformations  with $A$ to
those that preserve
$A_{\alpha_1}^{\infty}$ (it
gives certain unitary matrix
problems for
$A_{\alpha_2},\dots,
A_{\alpha_q}$ with partitions
them into blocks) and reduce
$A_{\alpha_2}$  to its
canonical form
$A_{\alpha_2}^{\infty}$, and
so on.  The obtained
representation
$$A^{\infty}=(A_{\alpha_1}^{\infty},\dots ,
A_{\alpha_q}^{\infty})$$ (we
omit the marking lines) will
be called a {\it canonical
representation} of the quiver
$Q$; the sequence of the
schemes $${\cal
S}(A^{\infty})=({\cal
S}(A_{\alpha_1}^{\infty}),\dots,
{\cal
S}(A_{\alpha_q}^{\infty}))$$
will be called the {\it
scheme} of $A^{\infty}$.

     Clearly, $A\simeq A^{\infty}$ and $A\simeq B$ if and only if
$A^{\infty}=B^{\infty}$.

\begin{theorem}\label{th_un3.1.2}
{\rm(a)} {\it Every unitary
representation is isometric
to a representation of the
form}
\begin{align*}
B&=(P_1\otimes
I_{m_1})\oplus\dots\oplus(P_t\otimes
I_{m_t})\\&
  \simeq\underbrace{P_1\oplus\dots\oplus P_1}_{\mbox{$m_1$ copies}}
\oplus\dots\oplus
\underbrace{P_t\oplus\dots\oplus
P_t}_{\mbox{$m_t$ copies}}
\end{align*}
{\it $($see
$(\ref{un-2.4}))$, where
$P_1,\dots, P_t$ are
nonisometric indecomposable
representations, uniquely
determined up to isometry,
and $m_1,\dots,m_t$ are
uniquely determined natural
numbers. Every autometry
$S:B\stackrel{\sim}{\to} B$
has the form
$$
S=({\bf 1}_{P_1}\otimes
U_1)\oplus\dots\oplus ({\bf
1}_{P_t}\otimes U_t),
$$
where $U_i$ is a unitary
$m_i\times m_i$ matrix $(1\le
i\le t)$}.

      {\rm(b)} {\it A unitary representation of dimension
$\ne (0,\dots,0)$ is
indecomposable if and only if
all its autometries are
scalar.}
\end{theorem}
                                                                      \medskip
\begin{proof}
Analogously (\ref{un-1.4}),
the matrices of every unitary
representation $A$ of $Q$ can
be accommodated in a block
diagonal matrix
$$A=\mbox{diag}(A_{\alpha_q},
A_{\alpha_{q-1}},\dots,A_{\alpha_1},0,\dots,0)$$
with a certain set of marked
blocks such that the
admissible transformations
with $A$ correspond to the
admissible transformations
with $M(A)$. Then
$M(A^{\infty})=
M(A)^{\infty}$ and we can
apply Theorem
\ref{th_un2.1.4}.\end{proof}


\subsection{The set of dimensions of indecomposable unitary representations}
\label{ss_un3.2} We will use
the following notation:

\begin{itemize}
  \item
$M_Q=[m_{ij}]$ is the
$p\times p$ matrix, in which
$m_{ij}$ is the number of
arrows $i\to j$ and $i\gets
j$ of the quiver $Q$;

  \item
supp$(z)$ is the full
subquiver of $Q$ with the
vertex set $\{i\,|\,z_i\ne
0\}$ for each $z\in {\mathbb
N}_0^p;$

  \item
$e_i=(0,\dots,1,\dots,0)\in{\mathbb
N}_0^p$ with 1 in the $i$th
position.
\end{itemize}

     Denote by $D(Q)$ the subset of ${\mathbb N}_0^p$
consists of $e_1,\dots, e_p$,
all $ e_i+e_j$ with $
m_{ij}=1$, and all nonzero
$z$ with connected
supp$(z)\notin \{\bullet,\
\bullet\! \to \!\bullet\}$
such that $zM_Q\ge
z$.\label{page_un_hyg}

     In this section we prove:                             \medskip

\begin{theorem}\label{th_un_3.2.1} {\it $D(Q)$ is the set
of dimensions of
indecomposable unitary
representations of a quiver
$Q$}. \end{theorem}

Put $$\Delta_i(z):=\sum_j
m_{ij}z_j,\qquad z\in{\mathbb
N}_0^p,\quad1\le i\le p. $$
Then
$$zM_Q=(\Delta_1(z),\dots,
\Delta_p(z)).$$
                                                           \medskip

\begin{lemma}\label{lem_un3.2.2}
 {\it  $D(Q)$ satisfies
the following conditions:}

     {\rm(i)}  {\it If $z\in D(Q)$ and ${\rm supp}(z)\notin \{\bullet,\
\bullet\! \to \!\bullet,\
\bullet
\unitlength 0.5mm
\linethickness{0.4pt}
\begin{picture}(8.0,2.00)(0.5,-1.8)
\put(5.53,0.00){\oval(4.00,4.00)[r]}
\bezier{16}(5.47,1.93)(3.67,2.00)(1.27,0.93)
\put(1.20,-0.67){\vector(-2,1){0.2}}
\bezier{20}(5.47,-2.00)(3.20,-1.87)(1.20,-0.67)
\end{picture}
\}$, then $zM_Q>z$.}

     {\rm(ii)}  {\it If $z,\,u\in D(Q)$ and $z<u$, then
there exists $i$ such that
$z+e_i\le u$ and $z+e_i\in
D(Q)$.}
\end{lemma}

\begin{proof}  (i)  Let
$$z\in D(Q),
\qquad \mbox{supp}(z)\notin
\{\bullet,\ \bullet\! \to
\!\bullet, \bullet
\unitlength 0.5mm
\linethickness{0.4pt}
\begin{picture}(8.0,2.00)(0.5,-1.8)
\put(5.53,0.00){\oval(4.00,4.00)[r]}
\bezier{16}(5.47,1.93)(3.67,2.00)(1.27,0.93)
\put(1.20,-0.67){\vector(-2,1){0.2}}
\bezier{20}(5.47,-2.00)(3.20,-1.87)(1.20,-0.67)
\end{picture}
\},\qquad zM_Q=z.$$ Fix $i$
such that $z_i= {\rm
max}\{z_1,\dots,z_p\}.$ Then
$m_{ij}\ne 0$ for a certain
$j\ne i$. Since
$$z_j=\Delta_j(z)\ge
m_{ij}z_i\ge z_i,$$ we have
$$z_i=z_j,\qquad m_{ij}=1,\qquad
m_{kj}z_k=0$$ for all $k\ne
i$. Taking $z_j$ and $z_i$
instead of $z_i$ and $z_j$,
we obtain $m_{ki}z_k=0$ for
all $k\ne j$. Hence
$\mbox{supp}(z)=\bullet\! \to
\!\bullet$, a contradiction.

     (ii)  Let  $z,\, u\in D(Q)$ and $z<u$.
 If supp$(z)\ne {\rm supp}(u)$, then there exists
a nonzero $m_{ij}$ with
$$i\in {\rm
supp}(u)\setminus{\rm
supp}(z),\qquad j\in{\rm
supp}(u)\cap {\rm supp}(z).$$
The $z+e_i$ satisfies the
requirements.

We may assume that $${\rm
supp}(z)={\rm supp}(u)=Q.$$
Then
$Q\not\in\{\bullet,\,\bullet\!
\to \!\bullet\}$.
 Fix a vertex $l$ such that $z_l<u_l$.
 We will suppose that $\Delta_l(z)=z_l$ and $m_{ll}=0$
(otherwise $z+e_l$ satisfies
the requirements).

     Assume first that $z_l\le z_j$ for some $m_{lj}\ne 0$.
The condition
$\Delta_l(z)=z_l$ implies
$z_l=z_j,\ m_{lj}=1,$ and
$m_{lk}=0$ for all $k\ne j$.
Hence $$z_j=z_l<u_l\le
\Delta_l(u)= u_j.$$ Since
$Q\ne \bullet\! \to
\!\bullet,\ m_{jk}\ne 0$ for
some $k\ne l$, and we can
take $z+e_j$.

     Next, let $z_l> z_j$ (and hence $z+e_j\in D(Q)$) for all nonzero
$m_{lj}$. If $z_j=u_j$ for
all $m_{lj}\ne 0$, then
$$u_l\le
\Delta_l(u)=\Delta_l(z)=z_l,$$
a contradiction. Hence
$z_j<u_j$ for a certain
$m_{lj}\ne 0$, and we can
take $z+e_j$. \end{proof}
                                                 \medskip

\begin{lemma}\label{lem_un3.2.3}
{\it If $A$ is a unitary $d$
representation of $Q$ and
$d\notin D(Q)$, then $A$
is decomposable.} \\[-3mm]                    
\end{lemma}

\begin{proof}  Assume to the
contrary, that $A$ is
indecomposable. Then
supp$(d)$ is connected; Lemma
\ref{lem_un2.1.1} and
$d\notin D(Q)$ imply $${\rm
supp}\,(d)\notin\{\bullet,\,\bullet\!
\to \!\bullet\},\qquad
dM_Q\ngeq d,$$ that is, there
exists $l$ such that
$\Delta_l(d)<d_l$. Then
$m_{ll}=0$ and we can assume
that there are no arrows
starting from $l$ (otherwise
we replace each arrow
$\alpha:l\to i$ by
 $\alpha^*: i\to l$, simultaneously replacing
$A_{\alpha} $ by the adjoint
matrix).

   Let $\alpha,\beta,\dots,\gamma$ be all arrows stopping at $l$;
combine the corresponding
them matrices of $A$ into a
single $d_l\times
\Delta_l(d)$ matrix
$$[A_{\alpha}|A_{\beta}|\cdots|A_{\gamma}].$$
 The number  of its rows is greater than the number
 of its columns; making a zero row by unitary
 transformations of rows, we obtain $A\simeq B\oplus P$,
where $P$ is the zero
representation of dimension
$e_l$, a contradiction.
\end{proof}
                                                     \medskip

\begin{lemma}\label{lem_un3.2.4}
{\it If there exists an
indecomposable unitary $z$
representation
 and $z_i<\Delta_i(z)$ for a certain vertex $i$, then there
exists an indecomposable
unitary  $z+e_i$
representation.}
\end{lemma}                                      \medskip

\begin{proof} Let
$A$ be an indecomposable
unitary $z$ representation
and $z_1<\Delta_1(z)$. We can
assume that each starting
from the vertex 1 arrow is a
loop (replacing each $\lambda
:1\to j,\ j\ne 1,$ by
$\lambda^*: j\to 1$ and,
respectively,
 $A_{\lambda}$ by $A_{\lambda^*}^*$).

     1) Assume first that there is  a loop $\alpha: 1\to 1$ and
define a unitary $z+e_1$
representation $H$ in which
\begin{itemize}
  \item
$H_{\alpha} $ is the
nilpotent Jordan block of
size $(z_1+1)\times (z_1+1)$,

  \item
$H_{\beta}:=A_{\beta}\oplus
0_{11}$ for each
 $\beta:1\to 1,\ \beta\ne\alpha$;

  \item
$H_{\gamma}:=
A_{\gamma}\oplus 0_{10}$ for
each $\gamma: j\to 1,\ j\ne
1$; and
  \item
$H_{\delta}:=A_{\delta}$ for
each $\delta:j\to k,\ k\ne
1$.
\end{itemize}
The representation $H$ is
indecomposable.

     Indeed, let $A^-$ and $H^-$ denote the restrictions of $A$ and $H$
on the subquiver
$Q^-:=Q\setminus\alpha.$ By
Theorem \ref{th_un3.1.2}(a),
 we may assume that
$$
A^-=(P_1\otimes
I_{m_1})\oplus\dots\oplus(P_t\otimes
I_{m_t}),
$$
where $P_1,\dots,P_t$ are
nonisometric indecomposable
representations of $Q^-,$
$P_1$ is the zero
representation of dimension
$e_1$, and $$ m_1\ge 0,\ \
m_2>0,\ \dots,\ m_t>0.$$
Clearly,
$$
H^-=(P_1\otimes
I_{m_1+1})\oplus (P_2\otimes
I_{m_2})\oplus\dots\oplus
(P_t\otimes I_{m_t}).
$$

Let $$S=(S_1,
S_2,\dots):H\stackrel{\sim}{\to}H.$$
Since
$S:H^-\stackrel{\sim}{\to}H^-$,
by Theorem
\ref{th_un3.1.2}(a)
                      \begin{align}
S&=({\bf 1}_{P_1}\otimes
U_1)\oplus\dots\oplus
({\bf 1}_{P_t}\otimes U_t),                  \notag  \\
S_1&=U_1^{(d_{11})}\oplus\dots\oplus
U_t^{(d_{1t})},
\label{un-3.3}
                       \end{align}
where
$$(d_{1j},\dots,d_{pj})=\mbox{dim}(P_j)$$
 and $U_j$ is a unitary matrix $(1\le j\le t)$.
Since
$$S_1H_{\alpha}=H_{\alpha}S_1,$$
$H_{\alpha}$ is a Jordan
block and $S_1$ is a unitary
matrix, we have $S_1=aI,\
a\in{\mathbb C}.$ The
representation $A$ is
indecomposable, so that
 $d_{1j}\ne 0$ and by (\ref{un-3.3})
$U_j=aI$ for all $1\le j\le
t.$ Hence $S=a{\bf 1}_H$ and
$H$ is indecomposable by
Theorem \ref{th_un3.1.2}(b).

2) There remains the case
$m_{11}=0$. Let
$$\alpha_1:j_1\to 1,\ \dots,\
\alpha_l:j_l\to 1$$ be all
the arrows stopping at 1. We
denote by $A^-$ the
restriction of $A$ on the
subquiver
$$Q^-:=Q\setminus\{1;
\alpha_1,\dots,\alpha_l\}.$$
By Theorem
\ref{th_un3.1.2}(a), we may
assume that
$$
A^-=(P_1\otimes
I_{m_1})\oplus\dots\oplus
(P_t\otimes I_{m_t}),
$$
where $P_1,\dots,P_t$ are
nonisometric indecomposable
unitary representations of
$Q^-$.

Let  $(S_2,\dots,S_p): A^-
\stackrel{\sim}{\to}A^-$. By
Theorem \ref{th_un3.1.2}(a),
$$
S_i=(I_{d_{i1}}\otimes
U_1)\oplus\dots\oplus
       (I_{d_{it}}\otimes U_t),
$$
where $(d_{2j},\dots,d_{pj})=
{\dim}P_j.$ For an arbitrary
unitary $z_1\times z_1$
matrix  $S_1$, we define
$\tilde A$ by means of
$$S=(S_1,
S_2,\dots,S_p):A\stackrel{\sim}{\to}{\tilde
A}$$ (then $\tilde A^-=A^-)$.
Taking into account that
$${\tilde A}_{\alpha_{\tau}}=
               S_1^{-1}
               A_{\alpha_{\tau}}S_{j_{\tau}}$$
and partitioning the sets of
columns of every
$A_{\alpha_{\tau}}$ and $
{\tilde A}_{\alpha_{\tau}}$
in the same manner as
$S_{j_{\tau}}$, we obtain
$$B:=[ A_{\alpha_1}|\cdots|
A_{\alpha_l}]=
        [B_1|\cdots|B_m]$$ and
$${\tilde B}:=[ {\tilde
A}_{\alpha_1}|\cdots| {\tilde
A}_{\alpha_l}]= [{\tilde
B}_1|\cdots|{\tilde B}_m],$$
where
$$m=\sum_{\tau=1}^l(d_{j_{\tau}1}+\dots
+d_{j_{\tau}t})$$ and
$${\tilde
B}_i=S_1^{-1}B_iU_{f(i)}$$
for a certain $f(i)\in
\{1,\dots,t\}.$

 Let $z_1\times u_i$ be the size of $B_i$ and put
$$r_i=\mbox{rank}[B_1|\cdots|B_{i-1}|B_{i+1}|\cdots|B_m].$$
$B$ is a $z_1\times
\Delta_1(z)$ matrix and $z_1<
\Delta_1(z)$, so
$z_1-r_i<u_i$ for a certain
$i$. Since $S_1$ and
$U_{f(i)}$ are arbitrary
unitary matrices, by Lemma
\ref{lem_un2.1.1}(a) there
exists $S$ such that
$${\tilde B}=\left[
              \begin{tabular}{c|c|c|c|c|c|c}
$C_1$&$\cdots$&$C_{i-1}$&$C_i$&$C_{i+1}$&$\cdots$&$C_m$\\
0&$\cdots$&0&$D$&0&$\cdots$&0
              \end{tabular}\right]\! ,
$$
where the rows of
$$[C_1|\cdots|C_{i-1}|C_{i+1}|\cdots|C_m]$$
are linearly independent and
$D$ is a $(z_1-r_i)\times
u_i$ matrix of the form
$$\mbox{diag}(a_1,\dots,a_n)\oplus
0_{kh}$$ with real
$a_1\ge\dots\ge a_n>0$. Since
${\tilde A}$ is
indecomposable and
$z_1-r_i<u_i$, we have  $k=0$
and $ h>0$.

     Let $a_{n+1}$ be a real number such that $a_n>a_{n+1}>0$.
The replacement $D$ by
$$D'=\mbox{diag}(a_1,\dots,a_n,a_{n+1})\oplus
0_{0,h-1}$$
changes ${\tilde B}$ to a new
matrix ${\tilde B}'$ and
${\tilde A}$ to a new
representation $H$ of
dimension $z+e_1$.

     Let  $R:H\stackrel{\sim}{\to}H$. Since $H$ and $A$
coincide on $Q^-$ and
$$(R_2,\dots,R_p): A^-
\stackrel{\sim}{\to}A^-,$$ by
Theorem \ref{th_un3.1.2}(a)
the matrices $R_2,\dots,R_p$
have the form
$$ R_j=(I_{d_{j1}}\otimes
V_1)\oplus\dots\oplus
(I_{d_{j\tau}}\otimes
V_\tau)$$ with unitary
$V_1,\dots,V_\tau$. By $$
R_1^{-1}{\tilde
B}'(R_{j_1}\oplus\dots\oplus
R_{j_l}) ={\tilde B}',$$
$R_1$ has the form
$R_{11}\oplus R_{12},$ where
$$R_{12}^{-1}D'V_{f(i)}=D'.$$
Lemma \ref{lem_un2.1.1}
implies
$$R_{12}=R_{13}\oplus [c].$$
Putting $${\tilde
R}_1=R_{11}\oplus
R_{13},\qquad {\tilde
R}_j=R_j\ (j>1),$$ we have
${\tilde R}:{\tilde
A}\stackrel{\sim}{\to}{\tilde
A}.$ By Theorem
\ref{th_un3.1.2}(b),
$${\tilde R}_j=aI,\qquad 1\le j\le
p,$$ for some $a\in{\mathbb
C}$, so $$V_j=aI,\qquad 1\le
j\le \tau.$$ In particular,
$V_{f(i)}=aI$ and, since
$$R_{12}^{-1}D'V_{f(i)}=D',$$
$c=a$ and $R_1=aI.$ Therefore
$R=a{\bf 1}_H$ and $H$ is
indecomposable by Theorem
\ref{th_un3.1.2}(b).
\end{proof}
\medskip

\begin{proof}%
[Proof of Theorem
\ref{th_un_3.2.1}] Let $U(Q)$
denote the
 set of dimensions of indecomposable unitary representations of $Q$.
Lemma \ref{lem_un3.2.3}
implies $U(Q)\subset D(Q)$.

     Let $u\in D(Q).$ Then $u_i\ne 0$ for a certain $i$.
Using Lemma
\ref{lem_un3.2.2}(ii), we
select a sequence
$$u_1:=e_i,\,u_2,\dots,u_t:=u$$
in $D(Q)$ such that
$$u_2-u_1,\dots,u_t-u_{t-1}\in\{e_1,\dots,e_p\}.$$
By Lemma \ref{lem_un3.2.4},
$$\{u_1,\dots,u_t\}\subset
U(Q),\qquad D(Q)\subset
U(Q).$$ \end{proof}


\subsection{The number of parameters in an indecomposable unitary
representation}
\label{ss_un3.3}
            By the {\it number of real (complex) parameters} of a
unitary representation $A$ we
mean the number of circles
(stars) in the scheme ${\cal
S}(A^{\infty})$.
 Recall that to circles correspond positive
real numbers in $A^{\infty}$,
and to stars correspond
complex numbers; the other
entries in $A^{\infty}$ are
zeros.

Kac \cite[Theorem C]{a_un7}
proved that the maximal
number of paremeters in an
indecomposable (non-unitary)
representation of dimension
$d$ over an algebraically
closed field is
$1-\varphi_Q(d)$, where
    $$
\varphi_Q(x)=x_1^2+\dots+x_p^2-
\sum_{i,j=1}^p m_{ij}x_ix_j
    $$
is a ${\mathbb Z}$-bilinear
form called the {\it Tits
form} of the quiver $Q,$ and
$m_{ij}$ is the number of
arrows $i\to j$ and $i\gets
j$.

We say that a zone (see page
\pageref{page_un_uyy}) is in
{\it general position} if all
its diagonal entries are
distinct and, if it is an
equivalence zone, nonzero. A
unitary representation $A$ is
said to be in {\it general
position} if all zones in
$A^{\infty}$ are in general
position.
\medskip

\begin{theorem}\label{th_un3.3.1}
{\rm(a)} {\it
 For every $d\in
D(Q)\ ($see page
\pageref{page_un_hyg}$)$
there exists an
indecomposable canonical
unitary $d$ representation of
general position, its scheme
is uniquely determined by
$d$.}

{\rm(b)} {\it An
indecomposable unitary $d$
representation $A$
 has $\sum d_i-1$ real parameters and at most
$$
1-\varphi_Q(d)+\frac 12\sum
d_i(d_i-1)
 $$
complex parameters; this
number is reached if and only
if $A$ is in general
position.}
\end{theorem}

\begin{proof} We consider the
set of zones of a canonical
unitary representation
$A^{\infty}=
(A_{\alpha_1}^{\infty},\dots
,A_{\alpha_q}^{\infty})$ as
linearly ordered:
\begin{align}\nonumber
Z_1<Z_2\ &\text{if}\
i_1<i_2;\ \text{or}\ i_1=i_2\
\text{and}\ l_1<l_2;\\&
\text{or}\ i_1=i_2,\ l_1=l_2\
\text{and}\
\mbox{Bl}(Z_1)<\mbox{Bl}(Z_2)
\label{un-3*}
\end{align}
(see (\ref{un-1.1}) and
Section \ref{ss_un2.2});
where $Z_k\ (k=1,2)$ is a
zone of depth $l_k$ in
$A_{\alpha_{i_k}}^{\infty}$.

(a) Let $d\in D(Q)$. By
Theorem \ref{th_un_3.2.1},
there exists an
indecomposable unitary $d$
representation $A$.
    Let $A$ be not in general position, and let $Z$ be the first
(in the sense  of
(\ref{un-3*})) zone of
$A^{\infty}$ that is not in
general position. Changing
diagonal entries of $Z$, we
transform it  into a zone
$\tilde Z$ of general
position and $A^{\infty}$
into a new representation
$\tilde A$.  This exchange
narrows down the set of
admissible transformations
that preserve all zones $\le
Z$, and, by Theorem
\ref{th_un3.1.2}(b),
$A^{\infty}$ has only scalar
autometries (as an
indecomposable
representation), therefore,
${\tilde A}$ has only scalar
autometries and is
indecomposable too.

If $\tilde A$ is not in
general position, we repeat
this process for it, and so
on, until we obtain an
indecomposable $d$
representation  $B$ of
general position. Its scheme
is uniquely determined since,
for each zone $Z$ of $B$, the
set of admissible
transformations that preserve
all zones $\le Z$ (and hence
the matrix problem for the
remaining part of $B$) does
not depend on diagonal
entries of $Z$ such that it
is in general position.

     (b) Let $A$ be an indecomposable canonical $d$ representation,
and $Z$ be its zone or the
symbol $\infty$. Denote by
$J(Z)$ the set of all
isometries of the form $S:
A\stackrel{\sim}{\to} {\tilde
A}$ that preserve all zones
$<Z$ (all zones if
$Z=\infty$). As follows from
the algorithms from pages
\pageref{page_un_154} and
\pageref{page_un_26468},
$J(Z)$ consists of all
sequences of the form
$S=(S_1,\dots,S_p),$ where
$$
S_i=U_{\sigma(i1)}\oplus
U_{\sigma(i2)}\oplus\dots\oplus
U_{\sigma(it_i)},
$$
$$\sigma: \{(ij)\,|\,1\le i\le
p,\ 1\le j\le t_i\}\to
\{1,\dots,t\}$$ is a fixed
surjection, and $U_1,\dots,
U_t$ are arbitrary unitary
   matrices of fixed sizes $m_1\times m_1,\dots,\ m_t\times m_t$
(we will write $S=S(
U_1,\dots, U_t))$.

     Put
$$
\Delta_1(Z)=m_1+\dots+m_t,\
\ \
\Delta_2(Z)=m_1^2+\dots+m_t^2.
$$

     Let $Z\ne \infty$ and $Z'$ be the first zone after $Z$
($Z'=\infty$ if $Z$ is the
last zone of $A$). We will
prove that
 \begin{align}
 \Delta_1(Z)- \Delta_1(Z')& =n^{\bullet}(Z),
 \label{un-3.3.1}\\
 \Delta_2(Z)- \Delta_2(Z')&\le 2n(Z)-n^{\bullet}(Z)-2n^{\star}(Z),
 \label{un-3.3.2}
 \end{align}
and that the equality in
\eqref{un-3.3.2} holds if and
only if $Z$ is a zone of
general position; where
$n(Z)$ is the number of
entries in $Z$, and
$n^{\bullet}(Z)$ (resp.,
$n^{\star}(Z)$) is the number
of circles (resp., stars)
that correspond to the
diagonal entries of $Z$.

As follows from the
algorithms from pages
\pageref{page_un_154} and
\pageref{page_un_26468}, the
block $\mbox{Bl}(Z)$ is
reduced by transformations
\begin{equation}
    \text{Bl}(Z)\mapsto U_i^{-1}\text{Bl}(Z)U_j,
    \label{un-3.3.3}
\end{equation}
where $$S=S(U_1,\dots,U_t)\in
J(Z)$$ and  $i$ and $j$ are
determined by $Z$; moreover,
this $S$ is contained in
$J(Z')$ if and only if
\eqref{un-3.3.3} preserves
$Z$.

(i) Let $i\ne j$, say, $i=1$
and $j=2$. Then, by Lemma
\ref{lem_un2.1.1},
$$Z=\text{Bl}(Z)=a_1I_{r_1}\oplus\dots\oplus
a_{k-1}I_{r_{k-1}} \oplus
0_{xy},\ r_{\alpha}\ge 1,$$
$x\ge 0$, and $y\ge 0$. The
transformation
\eqref{un-3.3.3} preserves
$Z$ if and only if
\begin{equation}
   U_1=V_1\oplus\dots\oplus V_k \ \ \text{and}  \ \
   U_2=V_1\oplus\dots\oplus V_{k-1}\oplus V_{k+1},
   \label{un-3.3.4}
\end{equation}
where $V_1,\dots,V_{k+1}$ are
unitary matrices of sizes
$$r_1\times
r_1,\dots,r_{k-1}\times
r_{k-1},\ x\times x,\ y\times
y.$$ Hence, $J(Z')$ consists
of all $S\in J(Z)$ with $U_1$
and $U_2$ of the form
\eqref{un-3.3.4}, that is,
$$S=S(V_1,\dots V_{k+1},
U_3\dots U_t).$$ Therefore,
$$\Delta_1(Z')=r_1+\dots+r_{k-1}+x+y+m_3+\dots+m_t,$$
$$\Delta_2(Z')=r_1^2+\dots+r_{k-1}^2+x^2+y^2+m_3^2+\dots
+m_t^2.$$ By
\eqref{un-3.3.3},
$\text{Bl}(Z)$ has size
$m_1\times m_2$,
$$m_1=r_1+\dots+r_{k-1}+x,\quad
m_2=r_1+\dots+r_{k-1}+y,$$ so
$$n(Z)=m_1m_2,\quad
n^{\bullet}(Z)=
r_1+\dots+r_{k-1},\quad
n^{\star}(Z)=0.$$ We have
$$\Delta_1(Z)-\Delta_1(Z')=
r_1+\dots+r_{k-1}=
n^{\bullet}(Z)$$ and
\begin{align*}
\Delta_2(Z)&-\Delta_2(Z')=
(r_1+\dots+r_{k-1}+x)^2\\&\qquad
+
   ( r_1+\dots+r_{k-1}
      +y)^2-r_1^2-\dots-r_{k-1}^2-x^2-y^2
   \\&=
   [(r_1+\dots+r_{k-1}+x)-( r_1+\dots+r_{k-1}+y)]^2
   \\&\qquad +
   2 (r_1+\dots+r_{k-1}+x)( r_1+\dots+r_{k-1}+y)
   \\&\qquad-
   r_1^2-\dots-r_{k-1}^2-x^2-y^2
  \\&=(x-y)^2+2n(Z)- r_1^2-\dots-r_{k-1}^2-x^2-y^2
  \\&=-2xy+2n(Z)- r_1^2-\dots-r_{k-1}^2
  \\&\le
  2n(Z)-r_1-\dots-r_{k-1}=2n(Z)-n^{\bullet}(Z).
\end{align*}
Moreover, we have the
equality if and only if
$$r_1=\dots=r_{k-1}=1,\qquad xy=0,$$ i.e, $Z$ is in
general position.

(ii) Let $i=j$, say, $i=j=1$.
Then, by Lemma
\ref{lem_un2.1.2},
$\text{Bl}(Z)=[F_{\alpha\beta}]$,
where
\begin{align*}
F_{\alpha\beta}&=0
\text{\quad if
$\alpha>\beta$, and}\\
F_{\alpha\alpha}&=
\lambda_{\alpha}I_{r_{\alpha}},\quad
r_{\alpha}\ge 1,\quad
r_1+\dots+r_k=m_1.
\end{align*}
The transformation
\eqref{un-3.3.3} preserves
$$Z=\{F_{\alpha\beta}\,|\,\alpha\le\beta\}$$
if and only if
$$U_1=V_1\oplus\dots\oplus
V_k,$$ where $V_1,\dots,V_k$
are unitary matrices of sizes
$$r_1\times r_1,\dots,\
r_k\times r_k.$$ Hence,
$J(Z')$ consists of all $S\in
J(Z)$ with $$U_1=
V_1\oplus\dots\oplus V_k,$$
that is, $$S=S(V_1,\dots,
V_{k}, U_2,\dots, U_t).$$ So
$$\Delta_1(Z)-\Delta_1(Z')=m_1-
r_1-\dots-r_k= 0
=n^{\bullet}(Z)$$ and
\begin{align*}
\Delta_2(Z)-\Delta_2(Z')&=
(r_1+\dots+r_k)^2
-r_1^2-\dots-r_k^2 \\&=2
\sum_{\alpha\le\beta}r_{\alpha}r_{\beta}
- 2(r_1^2+\dots+r_k^2)
\\&=2n(Z)- 2(r_1^2+\dots+r_k^2)
\\&\le 2n(Z)-2(r_1+\dots+r_k)
\\&= 2n(Z)-2n^{\star}(Z).
\end{align*}
Moreover, we have the
equality if and only if
$$r_1=\dots=r_k=1,$$ that is, $Z$
is in general position.

Hence, the relations
\eqref{un-3.3.1} and
\eqref{un-3.3.2} hold.

Let $Z_1<\dots<Z_r$ be all
zones of $A$ ordered by
(\ref{un-3*}), and let $$\dim
A=(d_1,\dots,d_p).$$ Then
$Z'_i=Z_{i+1}$ for $i<r$, and
$Z'_r=\infty.$ Since $J(Z_1)$
consists of all sequences
$S=(S_1,\dots,S_p)$ of
unitary $d_1\times d_1,\dots,
d_p\times d_p$ matrices,
$$\Delta_1(Z_1)=d_1+\dots+d_p,\quad
\Delta_2(Z_1)=d_1^2+\dots+d_p^2.$$
Since $A$ is indecomposable,
by Theorem
\ref{th_un3.1.2}(b)
$J(\infty)$ consists of all
sequences
$$S=\lambda(I_{d_1},\dots,
I_{d_p}),\qquad
\lambda\in{\mathbb C},\quad
|\lambda|=1,$$ so
$$S=S([\lambda]),\qquad
\Delta_1(\infty)=\Delta_2(\infty)=1.$$
By \eqref{un-3.3.1},
\begin{align*}
d_1+\dots+d_p-1&=\Delta_1(Z_1)-
\Delta_1(\infty)
\\&=\sum_{i=1}^r
(\Delta_1(Z_i)-
\Delta_1(Z'_i)) =\sum_{i=1}^r
n^{\bullet}(Z_i)
\end{align*}
is the number of circles in
$S(A^{\infty})$, that is, the
number of real parameters in
$A$.

By \eqref{un-3.3.2},
\begin{align*}
d_1^2+\dots&+d_p^2-1=\Delta_2(Z_1)-
\Delta_2(\infty)
\\&=\sum_{i=1}^r
(\Delta_2(Z_i)-
\Delta_2(Z'_i))
\\&\le
2\sum_{i=1}^r n(Z_i)-
\sum_{i=1}^r
n^{\bullet}(Z_i)-
2\sum_{i=1}^r n^{\star}(Z_i).
\end{align*}
 But $\sum_{i=1}^r n(Z_i)$ is the number of entries in $A_{\alpha_1},
\dots,A_{\alpha_q}$, hence,
it is equal to
$$\sum_{i,j=1}^p
m_{ij}d_id_j,$$ where
$m_{ij}$ is the number of
arrows $i\to j$ and
$i\leftarrow j$;
$$n^{\star}(A):=\sum_{i=1}^r
n^{\star}(Z_i)$$ is the
number of complex parameters
in $A$. Therefore,
\begin{align*}
n^{\star}(A) &\le\sum
m_{ij}d_id_j- \frac 12(\sum
d_i^2-1)- \frac 12(\sum
d_i-1)
 \\&=1-[\sum d_i^2-\sum
m_{ij}d_id_j] +\frac 12\sum
(d_i^2-d_i)
 \\&=
1-\varphi_Q(d)+\frac 12\sum
d_i(d_i-1).
\end{align*}
We have the equality if and
only if all $Z_i$ are in
general position, i.e., $A$
is in general position.
\end{proof}

The proof implies

\begin{corollary}\label{cor_un3.3.2}
{\rm(a)} { Let $d\in D(Q)$
and $m=\max
\{d_1,\dots,d_p\}$. Then
there exists an
indecomposable canonical $d$
representation of general
position with entries in
$\{0,\ 1,\dots,m\}$}.

{\rm(b)} {A decomposable
unitary $d$ representation
has less than $\sum d_i-1$
real parameters and less than
$$1-\varphi_Q(d)+\frac 12\sum
d_i(d_i-1)$$ complex
parameters.}
\end{corollary}

\begin{proof} (a) This statement
follows
from Theorems
\ref{th_un3.3.1}(a) and
\ref{th_un2.2.2}.

(b) This statement is proved
as Theorem
\ref{th_un3.3.1}(b), but, in
the last two paragraphs of
its proof, we must use
$\Delta_1(\infty)>1$ and
$\Delta_2(\infty)>1$ instead
of
$\Delta_1(\infty)=\Delta_2(\infty)=1$
since $J(\infty)$ contains a
non-scalar authometry in the
case of a decomposable
representation $A$.
\end{proof}


\section{Euclidean representations of a quiver}
\label{s_un4}
     Let $Q$ be a quiver with vertices $1,\dots,p$ and arrows
$\alpha_1,\dots,\alpha_q$. A
{\it Euclidean
representation} $A$ of
dimension
$$d=(d_1,\dots,d_p)\in
{\mathbb N}_0^p$$ will be
given by assigning a matrix
$A_{\alpha}\in{\mathbb
R}^{d_j\times d_i}$ to each
arrow $\alpha: i\to j;$ i.e.,
by the sequence
$$A=(A_{\alpha_1},\dots,
A_{\alpha_q}).$$ An ${\mathbb
R}$-{\it isometry}
$A\stackrel{\sim}{\to}_{\mathbb
R} B$ of Euclidean
representations $A$ and $B$
will be given by a sequence
$S=(S_1,\dots,S_p)$ of real
orthogonal matrices such that
$ S_jA_{\alpha}=B_{\alpha}S_i
$ for each arrow $\alpha:
i\to j$ (analogously,  $R:
A\stackrel{\sim}
{\to}_{\mathbb C} B$ denotes
an isometry in the sense of
Section 3).
A Euclidean $d$ representation $A$ 
is said to be ${\mathbb
R}$-{\it indecomposable} if
(i)~$d\ne(0,\dots,0)$ and
(ii) $A\simeq_{\mathbb R}
B\oplus C$ implies that $B$
or $C$ has dimension
$(0,\dots,0).$

For a sequence of complex
matrices
$$M=(M_1,\dots,M_n),$$ we
define the {\it conjugate
sequence} $${\bar M}=({\bar
M}_1,\dots,{\bar M}_n),$$ the
{\it transposed sequence}
$$M^T=(M_1^T,\dots,M_n^T),$$
and the {\it adjoint
sequence} $$M^*=\bar M^T.$$
Clearly, the Euclidean
representations are the
selfconjugate unitary
representations.


\subsection{A reduction to unitary representations}
\label{ss_un4.1}
 We give a standard reduction of the problem of classifying
Euclidean representations to
the problem of classifying
unitary representations.
\medskip

Let $\text{ind}(Q)$ and
$\text{ind}_{\mathbb R}(Q)$
denote complete systems of
nonisometric indecomposable
unitary representations and
non-${\mathbb R}$-isometric
${\mathbb R}$-indecomposable
Euclidean representations
respectively. Let us replace
each representation in
$\text{ind}(Q)$ that is
isometric to a Euclidean
representation by a Euclidean
one, and denote the set of
such by  $\text{ind}_0(Q)$
(if $A\in \text{ind}(Q)$ and
$S:
A{\stackrel{\sim}{\to}}_{\mathbb
C} {\bar A},$ then $A$ is
isometric to a Euclidean
representation if and only if
$S^T=S$; see Theorem
\ref{th_un4.2.2}). Denote by
$\text{ind}_1(Q)$ the set
consisting of all
representations from
$\text{ind}(Q)$ that are
isometric to their
conjugates, but not to a
selfconjugate, together with
one representation from each
pair
$\{A,\,B\}\subset\text{ind}(Q)$
such that
$$A{\not\simeq}_{\mathbb
C}\bar A\simeq_{\mathbb
C}B.$$

For a unitary $d$
representation
$$A=(A_{\alpha_1},
\dots,A_{\alpha_q}),$$ we
define the Euclidean $2d$
representation
 $$A^{\mathbb R}=(A_{\alpha_1}^{\mathbb R},
\dots,A_{\alpha_q}^{\mathbb
R}),$$ where
$A_{\alpha}^{\mathbb R}$, is
obtained from $A_{\alpha}$ by
replacing each entry $a+bi\
(a,b\in {\mathbb R})$ by the
block
$$
\begin{matrix}
 a & b \\
  -b&a
 \end{matrix}
$$
Since $$U^{-1}
\left[\genfrac{}{}{0pt}{}{a}{-b}\,
\genfrac{}{}{0pt}{}{b}{a}\right]
U=
\left[\genfrac{}{}{0pt}{}{a+bi}{0}
\genfrac{}{}{0pt}{}{0}{a-bi}\right]$$
with the unitary
$$U=\frac{1}{\sqrt{2}}
\left[\genfrac{}{}{0pt}{}{1}{i}\,
\genfrac{}{}{0pt}{}{-1}{i}\right],$$
we have
\begin{equation}
 A^{\mathbb R}\simeq_{\mathbb C}A\oplus{\bar A}.
 \label{un-4.1}
\end{equation}

\begin{theorem}\label{th_un4.1.1}
{\rm(a)} {Let $A$ and $B$ be
Euclidean representations of
a quiver $Q$. Then
$A\simeq_{\mathbb R} B$ if
and only if $A\simeq_{\mathbb
C} B$.}

{\rm(b)} {Every Euclidean
representation is $\mathbb
R$-isometric
 to a direct sum of indecomposable Euclidean representations, uniquely
determined up to $\mathbb
R$-isometry of summands.
Moreover,
\begin{equation}
{\rm ind}_{\mathbb R}(Q)=
{\rm ind}_0(Q)\cup
\{A^{\mathbb R}\,|\,A\in{\rm
ind}_1(Q)\}. \label{un-4.2}
\end{equation}}\par

{\rm(c)} {The set of
dimensions of ${\mathbb
R}$-indecomposable Euclidean
representations of $Q$
coincides with the set of
dimensions of indecomposable
unitary representations and
is equal to $D(Q)$ $($see
page
\pageref{page_un_hyg}$).$}
\end{theorem}

A {\it homomorphism}
(${\mathbb R}$-{\it
homomorphism}) $S:A\to B$ of
representations $A$ and $B$
of $Q$ is a sequence of
complex (real) matrices
$$S=(S_1,\dots,S_p)$$ such
that
$S_jA_{\alpha}=B_{\alpha}S_i$
for each arrow $\alpha:i\to
j$. Clearly, an isomorphism
$S$ is an isometry if and
only if $S^*=S^{-1}$.

\begin{lemma}\label{lem_un4.1.2}
 {The following properties are
equivalent for a unitary
$($Euclidean$)$
representation $A$:}

{\rm(i)} {$A$ is decomposable
$({\mathbb
R}$-decomposable$)$}.

{\rm(ii)} {\it There exists
an endomorphism $({\mathbb
R}$-{\it endomorphism}$)\
F:A\to A$ such that
$F=F^*=F^2\notin\{{\bf
0}_A,\,{\bf 1}_A\}.$}

{\rm(iii)} {There exists a
nonscalar selfadjoint
endomorphism
$({\mathbb R}$-{\it endomorphism}$)\ S=S^*: A\to A$.}\\[-3mm]    
\end{lemma}

\begin{proof}
(i)$\Rightarrow$(ii) Let $S:A
\stackrel{\sim}{\to} B\oplus
C$ be an isometry of unitary
representations (${\mathbb
R}$-isometry of Euclidean
representations) and $B\ne
0\ne C$. Then
$$F:=S^{-1}({\bf 1}_B\oplus
{\bf 0}_C)S:A\to A$$
satisfies (ii).

        (ii)$\Rightarrow$(iii) Put $S:=F$.

         (iii)$\Rightarrow$(i) Since every $S_i$ in $S$ is a Hermitian
(resp., real symmetric)
matrix, there exists a
unitary (real orthogonal)
matrix $U_i$ such that
$$R_i:=U_iS_iU_i^{-1}=\text{diag}(a_{i1},\dots,a_{it_i}),$$
where $a_{ij}\in {\mathbb R}$
and $$a_{i1}\ge\dots\ge
a_{it_i}.$$ Define the
unitary (Euclidean)
representation $B$ by means
of the isometry
$$U:=(U_1,\dots,U_p):A\stackrel{\sim}{\to}
B.$$ Then
$$R:=USU^{-1}:B\stackrel{\sim}{\to}
B$$ is an isometry and
$$R=a_1\mathbb
I_1\oplus\dots\oplus
a_t\mathbb I_t,$$ where
$$a_1>\dots>a_t,\quad \mathbb
I_i=(I_{n_{i1}},\dots,
I_{n_{ip}}),\quad n_{ij}\ge
0.$$ Clearly,
 $$B=B_1\oplus\dots\oplus
 B_t,$$
where
$$\dim(B_i)=(n_{i1},\dots,n_{ip}).$$
\end{proof}
                                                         \medskip
\begin{proof}[Proof of Theorem
\ref{th_un4.1.1}]   1) We
first prove the statement (a)
for an ${\mathbb R}$-{\sl
indecomposable} Euclidean
representation $ A$. Let
$$S=\Phi+i\Psi:
A\stackrel{\sim}{\to}_{\mathbb
C} B,$$ where $\Phi$ and
$\Psi$ are real matrices and
$B$ is a Euclidean
representation. Then $\Phi$
and $\Psi$ are $\mathbb
R$-homomorphisms $A\to B$.
Since
\begin{align*}
{\bf 1}_A
&=S^*S=(\Phi^T-i\Psi^T)(\Phi+i\Psi)
\\&=(\Phi^T\Phi+\Psi^T\Psi)+i(\Phi^T\Psi-\Psi^T\Phi),
\end{align*}
we have
$$
\Phi^T\Phi+\Psi^T\Psi={\bf
1}_A.
$$
By Lemma \ref{lem_un4.1.2},
the selfadjoint ${\mathbb
R}$-endomorphisms
$\Phi^T\Phi$ and $\Psi^T\Psi$
are scalar, that is,
$$\Phi^T\Phi=\lambda{\bf
1}_A,\quad \Psi^T\Psi=\mu{\bf
1}_A,\quad \lambda+\mu=1.$$
Obviously, $\lambda$ and
$\mu$ are non-negative real
numbers. For definiteness,
$\lambda>0,$ then
$\lambda^{-\frac 12} \Phi
:A\stackrel{\sim}{\to}_{\mathbb
R} B$.

      2) Let $A$ be an  ${\mathbb R}$-indecomposable Euclidean
representation that is
decomposable as a unitary
representation. We prove that
$$A\simeq_{\mathbb R}
B^{\mathbb R} \simeq_{\mathbb
C}B\oplus {\bar B},$$ where
$B$ is an indecomposable
unitary representation that
is not isometric to a
Euclidean representation.

Indeed, by Lemma
\ref{lem_un4.1.2} there
exists an endomorphism
$F:A\to A$ such that
$$F=F^*=F^2\notin\{{\bf
0}_A,\,{\bf 1}_A\}.$$ Let
$F=\Phi+i\Psi,$ where $\Phi$
and $ \Psi$ are sequences of
real matrices. Since
$$F=F^*=\Phi^T-i\Psi^T,$$ it
follows that $\Phi=\Phi^T$
and $\Psi=-\Psi^T$. By Lemma
\ref{lem_un4.1.2}, the
endomorphism $\Phi$ is
scalar, i.e.,
$\Phi=\lambda{\bf 1}_A,\
\lambda\in \mathbb R.$ If
$\lambda=0$, then
$$i\Psi=F=F^2=-\Psi^2$$ and
$\Psi={\bf 0}_A$, a
contradiction.

     Hence $\lambda\ne 0$. Since
     $$F=F^2=(\lambda{\bf 1}_A
+i\Psi)^2=(\lambda^2{\bf
1}_A-\Psi^2)+2\lambda
i\Psi,$$ we have
$$\lambda^2{\bf
1}_A-\Psi^2=\lambda {\bf
1}_A,\qquad 2\lambda\Psi
=\Psi.$$ The condition
$F\ne\bf 1_A$ implies
$$\Psi\ne {\bf 0}_A,\quad
\lambda=\frac 12,\quad
\Psi^2=-\frac14{\bf 1}_A.$$

By \cite[Sect. 4.4, Exercise
25]{a_un21}, every
nonsingular skew-symmetric
real matrix is real
orthogonally similar to a
direct sum of matrices of the
form
$$\left[\genfrac{}{}{0pt}{}{0}{-a}\,
\genfrac{}{}{0pt}{}{a}{0}\right],\qquad
a>0.$$   Since
$$\Psi^T=-\Psi,\qquad
\Psi^2=-\frac 14{\bf 1}_A,$$
there exists a sequence $S$
of real orthogonal matrices
such that $$S\Psi
S^{-1}=\frac 12{\mathbb
I}\otimes\matr{0}{1}{-1}{0}$$
(see \eqref{un-2.4}), where
${\mathbb I}=(I,\dots,I)$.
Put $$G:=SFS^{-1}=\frac
12{\mathbb
I}\otimes\matr{1}{i}{-i}{1}.$$

     Define the Euclidean representation $C$ by means of
$S
:A\stackrel{\sim}{\to}_{\mathbb
R} C$. Then $G:C\to C$ is an
${\mathbb R}$-endomorphism.
It follows from the form of
$G$ and the definition of
homomorphisms, that
$C=B^{\mathbb R}$ for a
certain $B$.

     If $B$ is a decomposable unitary representation, say,
$B\simeq_{\mathbb C}X\oplus
Y$, then by (\ref{un-4.1})
$$A\simeq_{\mathbb
R}B^{\mathbb R}
\simeq_{\mathbb C}B\oplus
{\bar B} \simeq_{\mathbb
C}X\oplus Y\oplus {\bar
X}\oplus {\bar Y}
\simeq_{\mathbb C}X^{\mathbb
R}\oplus Y^{\mathbb R},$$ by
1) $A\simeq_{\mathbb
R}X^{\mathbb R}\oplus
Y^{\mathbb R}$, a
contradiction.

     If $B$ is isometric to a Euclidean representation, say,
$$B\simeq_{\mathbb C}D={\bar
D},$$ then $$A\simeq_{\mathbb
R}B^{\mathbb R}
\simeq_{\mathbb C}B\oplus
{\bar B} \simeq_{\mathbb
C}D\oplus D,$$ by 1)
$A\simeq_{\mathbb R}D\oplus
D$, a contradiction. This
proves 2).

      (a)--(b). Let $A$ and $B$ be
Euclidean representations,
$A\simeq_{\mathbb R}B,$
$$A\simeq_{\mathbb
R}A_1\oplus\dots\oplus
A_l,\qquad B\simeq_{\mathbb
R}B_1\oplus\dots\oplus B_r,$$
where $A_i$ and $B_j$ are
${\mathbb R}$-indecomposable.
From 2) and Theorem
\ref{th_un3.1.2}(a), $l=r$
and, after a permutation of
summands, $A_i\simeq_{\mathbb
C}B_i$. By 1),
$A_i\simeq_{\mathbb R}B_i$.
The equality \eqref{un-4.2}
is obvious.

(c). By Corollary
\ref{cor_un3.3.2}(a), there
exists an
 ${\mathbb R}$-indecomposable Euclidean representation
 (with entries in ${\mathbb N}_0$) of
dimension $z$ for every $z\in
D(Q)$.
     Conversely, let $A$ be an  ${\mathbb R}$-indecomposable
Euclidean representation. If
$A$ is indecomposable as a
unitary representation, then
by Theorem \ref{th_un_3.2.1}
$\dim(A)\in D(Q)$. Otherwise
by 2) $A\simeq_{\mathbb
C}B\oplus{\bar B}$, where $B$
is an indecomposable unitary
representation, i.e.,
$d:=\dim(B)\in D(Q)$. Since
$B$ is not isometric to a
Euclidean representation,
$${\rm supp}\,(d)\notin
\{\bullet,\ \bullet\! \to
\!\bullet\}.$$ Applying twice
the definition of $D(Q)$ (see
page \ref{page_un_hyg}), we
have $$dM_Q\ge d,\quad
2dM_Q\ge 2d,\quad
\dim(A)=2d\in D(Q).$$
\end{proof}


\subsection{Unitary representations that are isometric to
                 Euclidean representations}
 \label{ss_un4.2}
Theorem \ref{th_un4.1.1}(b)
reduces the problem of
classifying Euclidean
representations of a quiver
$Q$ to the following two
problems:
\begin{itemize}
  \item
classify unitary
representations of $Q$ (i.e.,
construct the set ind$(Q)$);

  \item bring to light
for each
$A\in{\text{ind}}(Q)$ whether
it is isometric to a
Euclidean representation and
to construct that
representation.

\end{itemize}

In this section we consider
the second problem.
\medskip

\begin{lemma}\label{lem_un4.2.1}
{\rm(a)} {If $S$ is a
symmetric unitary matrix,
then there exists a unitary
matrix $U$ such that
$S=U^TU$}.

{\rm(b)} {If $S$ is a
skew-symmetric unitary
matrix, then there exists a
unitary matrix $U$ such that
$$S=U^T
\left(\matr{0}{1}{-1}{0}\oplus\dots\oplus
\matr{0}{1}{-1}{0}\right)
U.$$}
\end{lemma}

\begin{proof}
Analogous statement for a
non-unitary matrix $S$ is
given in \cite[Sect. 4.4,
Corollary 4.4.4 and Exercise
26]{a_un21}. The condition of
unitarity makes its proof
much more easy. We give it
sketchily since an explicit
form of $U$ is needed for the
applications of the next
theorem.

     Given a symmetric (skew-symmetric) unitary matrix
$S_n $ with rows
$s_1,\dots,s_n$. If $$s_1\ne
e_1:= (1,0,\dots,0),$$ we
take a unitary matrix $U_n$
with rows $u_1,\dots,u_n$
such that
$$u_1=\alpha(e_1+s_1),\quad
\alpha\in{\mathbb C},$$
(resp., ${\mathbb C}u_1+
{\mathbb C}u_2={\mathbb
C}e_1+ {\mathbb C}s_1$). Then
\begin{align*}
\bar u_1S_n&=\alpha(e_1+\bar
s_1)S_n =\alpha(e_1S_n+\bar
s_1S_n)\\&=\alpha(s_1+e_1)
=u_1=e_1U_n,
\end{align*}
hence
$$(U_n^{-1})^TS_nU_n^{-1}=\bar
U_nS_nU_n^{-1} =[1]\oplus
S_{n-1}$$ (resp., then $$\bar
U_n S_nU_n^{-1}
=\matr{0}{\beta}{-\beta}{0}\oplus
S_{n-2},$$ $|\beta|=1$;
replacing $u_2$ by $\beta
u_2$, we make $\beta=1$). If
$s_1=e_1$, we have
$$S_n=[1]\oplus
S_{n-1},\qquad U_n:=I_n.$$ We
repeat this procedure until
we obtain the required
$$U:=U_n(I_1\oplus
U_{n-1})(I_2\oplus U_{n-2})
\cdots(I_{n-1}\oplus U_1)$$
(resp., $U:=U_n(I_2\oplus
U_{n-2}) \cdots$).
\end{proof}
                                                        \medskip

\begin{theorem}\label{th_un4.2.2}
\begin{itemize}
  \item[\rm(a)]
 { Let $A$ be a
unitary representation and
$A\not\simeq_{\mathbb C}{\bar
A}$. Then $A$ is not
isometric to a Euclidean
representation.}

  \item[\rm(b)]    {Let $A$ be
an indecomposable unitary
representation and $S:
A\stackrel{\sim}{\to}_{\mathbb
C} {\bar A}.$}
\begin{itemize}
  \item[\rm(i)]
{\it If $S=S^T$, then $A$ is
isometric to a Euclidean
representation $B$ given by
$U:
A\stackrel{\sim}{\to}_{\mathbb
C} B$, where $U_1,\dots,U_p$
are arbitrary unitary
matrices such that
$U_i^TU_i=S_i\ ($they exist
by Lemma}
{\rm\ref{lem_un4.2.1}(a))}.

  \item[\rm(ii)]    {\it If
$S\ne S^T$, then $S=-S^T$ and
$A$ is not isometric to a
Euclidean representation but
is isometric to a unitary
representation $C$ of the
form $$\matr{X}{Y}{-\bar
Y}{\bar X}$$ given by $V:
A\stackrel{\sim}{\to}_{\mathbb
C} C$, where $V_1,\dots,V_p$
are arbitrary unitary
matrices such that $$V_i^T
\matr{0}{I}{-I}{0} V_i=S_i$$
$($they exist by Lemma}
{\rm\ref{lem_un4.2.1}(b))}.
\end{itemize}
\end{itemize}

\end{theorem}

\begin{proof}  (a) Let
$R:
A\stackrel{\sim}{\to}_{\mathbb
C}B$, where $B$ is a
Euclidean representation.
Then
$$R^T={\bar R}^{-1}: {\bar
B}\stackrel{\sim}{\to}_{\mathbb
C} {\bar A},\qquad H:=R^TR:
A\stackrel{\sim}{\to}_{\mathbb
C} {\bar A}$$ (observe that
$H=H^T$).

(b) Let $A$ be an
indecomposable unitary
representation and $S:
A\stackrel{\sim}{\to}_{\mathbb
C} {\bar A}$. Then ${\bar
S}S:
A\stackrel{\sim}{\to}_{\mathbb
C} A$, by Theorem
\ref{th_un3.1.2}(b)
$${\bar S}S=\lambda {\bf
1}_A,\ S=\lambda{\bar
S}^{-1}=\lambda S^T=\lambda
(\lambda S^T)^T =\lambda^2
S,$$ and
$\lambda\in\{1,-1\}$.

     (i) Let $\lambda=1,\ U: A\stackrel{\sim}{\to}_{\mathbb C}B$
and $U^TU=S$. Then
$$U=(U^T)^{-1}S={\bar U}S:
A\stackrel{\sim}{\to}_{\mathbb
C} {\bar B}$$ and $B={\bar
B}$.

     (ii) Let $\lambda=-1$. Then  $A$ is not
isometric to a Euclidean
representation (otherwise, by
(a) there exists $$H=H^T:
A\stackrel{\sim}{\to}_{\mathbb
C}\bar A;$$ by Theorem
\ref{th_un3.1.2}(b)
$$S^{-1}H=\mu {\bf
1}_A,\qquad H^T=\mu S^T=-\mu
S=-H,$$ a contradiction).
      Let $V: A\stackrel{\sim}{\to}_{\mathbb C}C$,
where $$V_i^T
\matr{0}{I}{-I}{0} V_i=S_i.$$
Then $$\bar VSV^{-1}:
 C\stackrel{\sim}{\to}_{\mathbb C}{\bar
 C}.$$
If $\alpha $ is an arrow of
$Q$, then
$$\matr{0}{I}{-I}{0}C_{\alpha}={\bar
C}_{\alpha}
\matr{0}{I}{-I}{0}$$ and
$C_{\alpha}$ is of the form
$$\matr{X}{Y}{-\bar Y}{\bar
X}.$$ \end{proof}
                                                      \medskip

Applying this theorem to
unitary representations of
the quiver \quiv\!\!, we
obtain

\begin{corollary}\label{col_un4.2.2}
{Let $A$ be a complex matrix
that is not unitarily similar
to a direct sum of matrices,
and let $S^{-1}AS={\bar A}$
for a unitary matrix $S\
($such $S$ exists if $A$ is
unitarily similar to a real
matrix$)$. Then $A$ is
unitarily similar to a  real
matrix if and only if $S$ is
symmetric.} \eprf
\end{corollary}

\end{document}